\documentclass[reqno,12pt,a4paper]{amsart}
\usepackage{pgfplots}
\usepackage[english]{babel}
\usepackage{amsmath,amsthm,amssymb,amsfonts, mathdots, framed}
\usepackage{scrtime, cancel}
\usepackage{enumitem, color, cases}
\usepackage{calrsfs}
\DeclareMathAlphabet{\pazocal}{OMS}{zplm}{m}{n}

\usepackage[immediate,debrief]{silence}
\usepackage{hyperref}

\usepackage{mathtools}
\usepackage{pdfsync,verbatim}
\usepackage{color}
\mathtoolsset{showonlyrefs}

\numberwithin{equation}{section}   

\allowdisplaybreaks
\newtheorem{theorem}{Theorem}[section]
\newtheorem{lemma}[theorem]{Lemma}
\newtheorem{proposition}[theorem]{Proposition}
\newtheorem{corollary}[theorem]{Corollary}

\theoremstyle{definition}
\newtheorem{remark}[theorem]{Remark}

\newcommand{\cone}{\mathcal C}
\newcommand{\R}{\mathbb{R}}
\newcommand{\C}{\mathbb{C}}
\newcommand{\N}{\mathbb{N}}

\newcommand{\ellipses}{E}
\def\misgausskd{d \gamma_\infty}
\def\misgaussk{\gamma_\infty}
\DeclareMathOperator{\tr}{tr}
\newcount\dotcnt\newdimen\deltay
\def\Ddot#1#2(#3,#4,#5,#6){\deltay=#6\setbox1=\hbox to0pt{\smash{\dotcnt=1
\kern#3\loop\raise\dotcnt\deltay\hbox to0pt{\hss#2}\kern#5\ifnum\dotcnt<#1
\advance\dotcnt 1\repeat}\hss}\setbox2=\vtop{\box1}\ht2=#4\box2}

\newcommand{\bigO}{\pazocal{O}}

\usepackage[utf8]{inputenc}

\newcommand\nucleorieszfirstorder{\mathcal{R}_j }
\newcommand\nucleorieszsecondorder{\mathcal{R}_{ij}}
\newcommand\nucleorieszfirstorderloc{ \mathcal{R}_{j,0} }
\newcommand\nucleorieszfirstorderinf{ \mathcal{R}_{j,\infty} }
\newcommand\nucleorieszsecondorderloc{\mathcal{R}_{ij,0}}
\newcommand\nucleorieszsecondorderinf{\mathcal{R}_{ij,\infty}}

\makeatletter
\def\author@andify{%
  \nxandlist {\unskip ,\penalty-1 \space\ignorespaces}%
    {\unskip {} \@@and~}%
    {\unskip \penalty-2 \space \@@and~}%
}
\makeatother

\title[Riesz transforms ]{
Riesz transforms  \\of a general Ornstein--Uhlenbeck  semigroup }

\author{Valentina Casarino}
\address{Universit\`a degli Studi di Padova\\Stradella san Nicola 3 \\I-36100 Vicenza \\ Italy}
\email{valentina.casarino@unipd.it}
\author{Paolo Ciatti}
\address{Universit\`a degli Studi di Padova\\Via Marzolo 9 \\I-35100 Padova \\ Italy}
\email{paolo.ciatti@unipd.it}
\author{Peter Sj\"ogren}
\address{Mathematical Sciences,  University of Gothenburg and  Mathematical Sciences,
Chalmers University of Technology  \\ SE - 412 96 G\"oteborg, Sweden}
\email{peters@chalmers.se}


\thanks{The first two authors were partially supported by GNAMPA (Project 2018
``Operatori e disu\-gua\-glianze integrali in spazi con simmetrie")
and MIUR (PRIN 2016 ``Real and Complex Manifolds: Geometry, Topology and Harmonic Analysis").
The third author was supported by GNAMPA (Professore Visitatore Bando 30/11/2018).
This research was carried out while the third author was visiting the University of Padova,  Italy, and he is grateful for its hospitality.}
\date{\today, \thistime}

\begin{document}

\begin{abstract}
We consider Riesz transforms of any order
associated to an Ornstein--Uhlenbeck operator            
with covariance                                  
 given by a real, symmetric and positive definite matrix,
and with drift                                            
given by a real   matrix
whose eigenvalues have negative real parts.
In this general Gaussian context, we prove that a Riesz transform
is of weak type $(1,1)$ with respect to the invariant measure if and only if
its order is at most $2$.

\end{abstract}

\maketitle

\section{Introduction}\label{intro}

In this  paper
we are concerned with Riesz transforms of any order in a general Gaussian setting,
 in  $\R^n$ with $n\ge 1$.
An Ornstein-Uhlenbeck semigroup is determined by
two  $n\times n$ real matrices $Q$ and $B$ such that
\begin{enumerate}
\item[(h1)]  $Q$ is symmetric and positive definite;
\item[(h2)] all the  eigenvalues of $B$   
have negative real parts.
\end{enumerate}
Here
 $Q$  and $B$ indicate the covariance and the drift, respectively.
We also introduce a family of covariance  matrices
\begin{equation}\label{defQt}
Q_t=\int_0^t e^{sB}\,Q\,e^{sB^*}\,ds, \qquad \text{ $t\in (0,+\infty]$}.
\end{equation}
Observe that these $Q_t$, including $Q_\infty$, are well defined, symmetric
 and positive definite.
Then we define a family of normalized Gaussian measures in $\R^n$,
$$
d\gamma_t (x)=
(2\pi)^{-\frac{n}{2}}\,
(\text{det} \, Q_t)^{-\frac{1}{2} }\,
e^{-\frac12 \langle Q_t^{-1}x,x\rangle}
\,dx  \,,\qquad \text{ $t\in (0,+\infty]$}.$$
On the space
of bounded continuous functions in $\R^n$,
 the
Ornstein--Uhlenbeck semigroup
 is explicitly given
 by  Kolmogorov's formula \cite{Kolmog, DaPrato}
\begin{equation*}
\mathcal H_t
f(x)=
\int
f(e^{tB}x-y)\,d\gamma_t (y)\,, \; \qquad x\in\R^n,  \; \quad t>0,
\end{equation*}
and generated  by the Ornstein--Uhlenbeck  operator, defined below.
Notice that
$d\gamma_\infty$ is  the unique invariant
probability
 measure with respect to the semigroup  
$\left(
\mathcal H_t
\right)_{t> 0}$; its
density is proportional to
$e^{-R(x)}$, where $R(x)$ denotes the quadratic form
\begin{equation*}
R(x) ={\frac12 \left\langle Q_\infty^{-1}x ,x  \right\rangle}, \qquad\text{$x\in\R^n$}.
\end{equation*}
In this general Gaussian framework,
the  Ornstein--Uhlenbeck operator $\mathcal L$ is given
for functions $f\in
\mathcal S (\R^n)$ by
$$\mathcal L f(x)=
\frac12\,
\tr
\big(
Q\nabla^2
f\big)(x)+
\langle Bx, \nabla f(x) \rangle
\,,\qquad
x\in\R^n,
$$
where $\nabla$ is the gradient and $\nabla^2$ the Hessian.
Notice that   $-\mathcal L$ is elliptic.
We write $D=(\partial_{x_1}, \ldots, \partial_{x_n})$ in $\R^n$ and let
$\alpha=(\alpha_1,\ldots, \alpha_n) \in \N^n\setminus \{(0,\dots,0) \}$
denote a multiindex, of length  $|\alpha| = \sum_1^n \alpha_i $.
Then we can  define
the Gaussian Riesz transforms
as
\begin{equation*}
 R^{(\alpha)} = D^\alpha
    (-\mathcal L)^{-|\alpha|/2} \ P_0^{\perp},
 \end{equation*}
where
$P_0^{\perp}$ is the orthogonal projection  onto the orthogonal  complement in $L^2(\gamma_\infty)$ of the eigenspace corresponding to the eigenvalue $0$.
Here
the derivatives  are taken in the sense of distributions.
We will justify the introduction of negative powers of $-\mathcal L$
in Section \ref{negpowers}.

When the order $|\alpha|$ of $R^{(\alpha)}$  equals $1$ or $2$, we shall denote
by  $R_j $ and
$R_{ij} $  the corresponding
Riesz transforms,
that is, for $i,j\in\{1,\ldots, n\}$
\begin{equation*}
 R_j=
{\partial_{ x_j}}
  (-\mathcal L)^{-1/2}  \ P_0^{\perp}
 \end{equation*}
and
\begin{equation*}
 R_{ij}=
{\partial_{ x_i  x_j}}
  (-\mathcal L)^{-1}  \ P_0^{\perp}
.
 \end{equation*}

\medskip

There exists a vast literature concerning the $L^p$ boundedness of Riesz transforms in the Gaussian  setting, in
both  the strong and the weak sense.
We will only mention the results that are most significant for this work;
here $1<p<\infty$.

In the standard case,
when $Q$ and $-B$ are  the identity matrix,
the strong type $(p,p)$
of  $ R^{(\alpha)}$
has been proved with different techniques  in
\cite{Meyer, Gundy, Pisier,  Urbina, Feyel, FSU, {SoriaCR}};
for a recent account of this case we refer to \cite[Chapter 9]{Urbina-monograph}.
Other proofs, holding in the more general case
$Q=I$ and $B$ symmetric,
may be found in \cite{Gut, GST}.
G. Mauceri and L. Noselli have shown more recently
that the Riesz transforms of any order are bounded on $L^p(\gamma_\infty)$
in the general case
 (see  \cite[Proposition 2.3]{Mauceri-Noselli}).
For some  results  in an infinite-dimensional framework, we refer to
\cite{Goldys}.

The problem of the weak type $(1,1)$ of $  R^{(\alpha)}$ is more involved than  in the Euclidean context,
where it is well known that  a Riesz transform of any order associated to the Laplacian is of weak type $(1,1)$.
Indeed, in the standard Gaussian framework   $Q=-B=I$,
 it is known that $ R^{(\alpha)}$ is  of weak type $(1,1)$ if and only if $|\alpha|\le 2$
(see \cite{Muckenhoupt2, ScottoPhd,
FS,
Scotto, FGS,  PerezJGA, Perez, Soria-Perez, GCMST, FHS, Bruno} for different proofs).
 In their paper  \cite{Mauceri-Noselli},
Mauceri and Noselli proved the weak type $(1,1)$  of the first-order Riesz transforms associated to
an   Ornstein--Uhlenbeck semigroup
 with covariance $Q=I$ and
drift $B$
satisfying a certain technical condition.
To the best of our knowledge, no result beyond this is known  about the weak type (1,1), neither for
first-order Riesz transforms associated to  more general   semigroups nor for higher-order Riesz operators.
\medskip

In this paper we continue the analysis started in \cite{CCS1} and \cite{CCS2}
of a general  Ornstein--Uhlenbeck semigroup, with
 real matrices $Q$ and $B$ satisfying only
 (h1) and (h2).
Our main result will be the following extension of the result in the standard
case.

\begin{theorem}\label{weaktype1}
The Riesz transform $R^{(\alpha)}$
  associated to the Ornstein--Uhlenbeck operator $\mathcal L$
is of weak type $(1,1)$ with respect to the invariant measure $d\gamma_\infty$
if and only if  $|\alpha |\le 2$.
\end{theorem}
In particular, we shall prove the inequalities
\begin{equation} \label{thesis-mixed-Di}
\gamma_\infty
\{x\in\R^n : R_j\,
f(x) > C\lambda\} \le \frac{C}\lambda\,\|f\|_{L^1( \gamma_\infty)},\qquad \text{ $\lambda>0$,}
\end{equation}
and
\begin{equation} \label{thesis-mixed-Di-2}
\gamma_\infty
\{x\in\R^n : R_{ij}\,f(x) > C\lambda\} \le \frac{C}\lambda\,\|f\|_{L^1( \gamma_\infty)},\qquad \text{ $\lambda>0$,}
\end{equation}
 for all $i,\:j=1,\ldots,n$ and all  functions $f\in L^1 (\gamma_\infty)$,
with $C=C(n,Q,B)>0$.

\medskip

The plan of the paper is as follows.
In  Section
\ref{s:particular},
we introduce the Mehler kernel $K_t(x,u)$, which is the integral kernel of $\mathcal H_t$.
Some estimates of this kernel are also given.
 As in  \cite{CCS2}, we introduce a system of polar coordinates which is essential in our approach, and we define suitable global and local regions.
Section \ref{negpowers} deals with the
 definition of the negative powers of $-\mathcal L$.
 
Then in Section \ref{Riesz transforms}
we   explicitly write    the kernels of  $R_{j}$ and
$R_{ij}$ as
 integrals with respect to the parameter $t$, taken over  $0<t<+\infty$.
Section \ref{Some estimates for large t} contains  bounds for  those parts of
these kernels  which are given by  integrals only over  $t>1$.
In Section \ref{reduction},
several  technical simplifications  reducing the complexity of the proof
of  Theorem \ref{weaktype1}
are discussed.
After this preparatory work,
the proof of the theorem,    
which  is quite involved and requires several steps,
begins. In Section~\ref{larget},
 we consider those parts corresponding to  $t>1$  of
 the kernels of  $R_{j}$ and
$R_{ij}$, and prove a weak type estimate.
Section~\ref{ The local case  }
 is devoted to
 the  proof of
  the weak bounds
for the local parts of the operators.
Finally, in Section~\ref{smallt-global}
we
conclude the proof of the sufficiency part of Theorem \ref{weaktype1},
by proving the weak type
estimates for the global parts, with the integrals restricted to
$0<t<1$.
In Section~\ref{Counterex},
we establish the necessity statement in Theorem  \ref{weaktype1} by means of a
 counterexample.

\medskip

In the following,
the symbols $c>0$ and $C<\infty$
will denote various constants, not necessarily the same at
different
occurrences. All of them  depend only on the dimension $n$ and on $Q$ and $B$.
With $a, \:b > 0$ we write $a\lesssim b$  instead of $a \leq C b$
and $a \gtrsim b$ instead of $a \geq c b$.
The relation $a\simeq b$ means that both $a\lesssim b$ and
$a \gtrsim b$  hold.

By $\N$ we denote the set of all nonnegative integers.
 If $A$ is  an $n\times n$ matrix, we write $\|A\|$
for its operator norm on $\R^n$
with the Euclidean norm $|\cdot|$.
We let
\[
|x|_Q = | Q_\infty^{-1/2}x|,
\]
so that $R(x) = |x|_Q^2/2$. Observe that $|x|_Q$
 is a norm on $\R^n$ and that  $|x|_Q \simeq |x|$.

Integral kernels of operators are always meant in the sense of integration with respect to
the measure $d\gamma_\infty$.

\section{ Notation and preliminaries
}\label{s:particular}

It follows from \eqref{defQt}
that for $0<t<\infty$
\begin{equation*}
Q_\infty -Q_t =\int_t^{\infty}e^{sB}Qe^{sB^*}ds.
\end{equation*}
This difference and also
\begin{equation*}
Q_t ^{-1} -  Q_\infty^{-1} = Q_t^{-1}( Q_\infty -Q_t )Q_\infty^{-1}
\end{equation*}
are  symmetric and strictly positive-definite matrices.

\medskip

It is shown in \cite[formula (2.6)]{CCS2} that
for bounded and continuous functions $f$
\begin{align*}
 \mathcal H_t
f(x) &=
 \int
K_t
(x,u)\,
f(u)\,
 d\gamma_\infty(u),  \qquad t>0,
\end{align*}
where  the Mehler  kernel $K_t$ is
given by
\begin{align}\label{defKR}
K_t (x,u)
=
\Big(
\frac{\text{det} \, Q_\infty}{\text{det} \, Q_t}
\Big)^{{1}/{2} }
 e^{ R(x)}\,
\exp \left[
{-\frac12
\left\langle (
Q_t^{-1}-Q_\infty^{-1})\, (u-D_t x) \,,\, u-D_t x\right\rangle}\right]
\qquad ~
\end{align}
for $x,u\in\R^n$ and $t>0$. Here
we  use
a one-parameter group of matrices
\begin{equation*}
 D_t =
 Q_\infty\,
 e^{-tB^*}\, Q_\infty^{-1} , \qquad \text{$t\in\R $.
}
\end{equation*}
We recall from \cite[Lemma 2.1]{CCS2}
that $D_t$ may be expressed in various ways.
 Indeed, for   $t>0$   
one has
\begin{align}
D_t &=
(Q_t^{-1}-Q_\infty^{-1}
)^{-1}\, Q_t^{-1}\,
e^{tB}  \label{def:Dtx}
 \end{align}
and
\begin{align}
D_t &= e^{tB} + Q_t\, e^{-tB^*}\,Q_\infty^{-1}. \label{Dtxii}
 \end{align}
We restate
Lemma 3.1 in \cite{CCS2}.
\begin{lemma}\label{expsB-bounded}
For $s>0$ and for all $x\in \R^n$
the matrices $ D_{s }$ and $ D_{-s }= D_{s }^{-1}$ satisfy
\begin{equation*}
  e^{cs}|x| \lesssim |D_s\, x|  \lesssim   e^{Cs} |x|,
\end{equation*}
and \begin{equation*}
  e^{-Cs}|x| \lesssim |D_{-s}\, x|  \lesssim   e^{-cs} |x|.
\end{equation*}
This also holds with  $ D_{s }$ replaced by  $e^{-sB}$ or  by  $e^{-sB^*}$.
 \end{lemma}

The following is part of  \cite[Lemma 3.2]{CCS2}.
\begin{lemma}\label{stimadet}
For all $t>0$ one has
\begin{enumerate}[label=(\roman*)]
\item[{\rm{(i)}}]
$\det{ \, Q_t}
\simeq
(\min(1,t))^{n}$;
\item[{\rm{(ii)}}]
$\|
Q_t^{-1}\|\simeq (\min (1,t))^{-1}$;
\item[{\rm{(iii)}}]
$\|
Q_t^{-1}-Q_\infty^{-1}
\|\lesssim {t}^{-1}\,{e^{-ct}}$;
\item[{\rm{(iv)}}]
$\|\left(
Q_t^{-1}-Q_\infty^{-1}
\right)^{-1/2}\|\lesssim {t^{1/2}}\, e^{Ct}$.
\end{enumerate}
 \end{lemma}

Lemma 4.1 in \cite{CCS2} says
that for all $x \in \R^n$ and $s\in\R$ one has
 \begin{align}  \label{vel-3}
&\frac{\partial}{\partial s}\,
 D_s\,x
 =
 -Q_\infty\, e^{-sB^*}B^*\,Q_\infty^{-1} x;
\\
&\frac{\partial}{\partial s}\,
R\big( D_s\,x \big)
\simeq
\big|
D_s\, x\big|^2 .
 \label{vel-4}
\end{align}
In \eqref{vel-3} we can estimate $e^{-sB^*}$
by means of Lemma \ref{expsB-bounded}, to get
\begin{equation} \label{dDs}
  \left|\frac{\partial}{\partial s}\,
 D_s\,x \right| \simeq |x|,  \qquad
|s|\le 1.
\end{equation}
Integration of  \eqref{vel-4} leads to
\begin{equation}
  \label{stime per R-t(x)-R(x)}
|R(D_t \,x)-R(x)| \simeq |t|\,    |x|^2, \qquad
|t|\le 1,
\end{equation}
again because of Lemma  \ref{expsB-bounded}.

\begin{lemma} \label{differ}
  Let $0 \ne x \in \R^n$ and $|t| \le 1$. Then
  \begin{equation*}
    |x- D_t \,x| \simeq |t|\, |x|.
  \end{equation*}
\end{lemma}
\begin{proof}
The upper estimate is an immediate consequence of \eqref{dDs}.
  For the lower estimate, we write
  \begin{align}\label{x-Dtx}
    |x- D_t\, x| &\simeq
|x- D_t \,x|_Q \ge \big| |x|_Q -|D_t\, x|_Q \big| \notag  \\ &=
\frac{\left|  |x|_Q^2 -|D_t\, x|_Q^2 \right|}{ |x|_Q +|D_t \,x|_Q}
 \simeq
\frac{|t|\, |x|^2}{|x|_Q} \simeq |t|\, |x|,
  \end{align}
where we used \eqref{stime per R-t(x)-R(x)} to estimate the numerator
and Lemma  \ref{expsB-bounded} for the denominator.
\end{proof}

The following
 implication will be useful as well.
Since $R(x) = |x|_Q^2/2$ and $|.|_Q$ is a norm,
\begin{equation}\label{diff-R}
R(x)>2R(y)  \qquad   \Rightarrow  \qquad   
R(x-y)\simeq R(x).
\end{equation}

\medskip

We finally give estimates of the kernel $K_t$, for small and large values of
$t$.
     Combining   \eqref{defKR} with
       Lemma~\ref{stimadet} (iii) and (iv), we have
\begin{equation}\label{litet}
   \frac{ e^{R( x)}}{t^{n/2}}\exp\left(-C\,\frac{|u-D_t \,x |^2}t\right)
 \lesssim   K_t(x,u)
\lesssim  \frac{ e^{R( x)}}{t^{n/2}} \exp\left(-c\,\frac{|u-D_t\, x |^2}t\right),
 \; \quad 0 < t\le 1.
\end{equation}

For  $t\geq 1$, we can use the norm $|.|_Q$ to write
             \cite[Lemma 3.4]{CCS2}
             slightly more precisely.
The proof of   \cite[Lemma 3.3]{CCS2}
shows that
\[\left\langle (Q_t^{-1}-Q_\infty^{-1})\,D_t\, w, D_t\, w\right\rangle \ge |w|^2_Q\]
for any $w \in \R^n$, and this leads to
\begin{equation} \label{tstort}
e^{R(x)}
\exp
\Big[
-C
\left|
D_{-t}\,u- x\right|_Q^2
\Big]
\lesssim
K_t (x,u)
\lesssim e^{R(x)}
\exp
\Big[
-\frac12
\left|
D_{-t}\,u- x\right|_Q^2
\Big],  \quad   \,   t\geq 1.
\end{equation}

 For
 $\beta>0$,  let
 $E_\beta$ be the ellipsoid
\begin{equation*}
\ellipses_\beta
=\{z\in\R^n:\, R(z)=
\beta\}
\,.\end{equation*}
As in     \cite[Subsection 4.1]{CCS2},
we introduce polar coordinates $(s, \tilde x)$  for any point $x\in\R^n,\,\, x\neq 0$,
 by writing
\begin{equation}\label{def-coord}
x=D_s\, \tilde x
\end{equation}
with $\tilde x\in \ellipses_\beta$
and $s\in\R$.

 The Lebesgue measure in $\Bbb R^n$ is given in terms of
$(s, \tilde x)$ by
\begin{align}\label{def:leb-meas-pulita}
  dx =
e^{-s\tr B}\, \frac{ |Q^{1/2}\, Q_\infty^{-1} \tilde x |^2}
{2\,| Q_\infty^{-1} \tilde x  |}\,
ds\, dS_\beta( \tilde x)
\simeq e^{-s\tr B}\, | \tilde x|\,ds\, dS_\beta( \tilde x),
\end{align}
where  $dS_\beta$ denotes the area measure of
  $\ellipses_\beta$.
We refer to \cite[Proposition 4.2]{CCS2} for a proof.

\medskip

 For any $A>0$
we define  global and local regions
\begin{align*}
G_A&=\left\{
(x,u)\in\R^n\times\R^n\,:  \,|x-u|> \frac{A}{1+|x|}
\right\}
\end{align*}
and
 \begin{align*}
L_A&=\left\{
(x,u)\in\R^n\times\R^n\,:  \,|x-u|\le \frac{A}{1+|x|}
\right\}.
\end{align*}

\section{\texorpdfstring{On the
 definition of negative powers of     
            $-\mathcal L$
             }{On the definition of negative powers of -L
             }}
\label{negpowers}

We start recalling the definition of Riesz transforms
introduced by Mauceri and Noselli
in \cite{Mauceri-Noselli} in a nonsymmetric context.
For any nonzero multiindex $\alpha\in\mathbb N^n$,
the Riesz transform $ R^{(\alpha)}$,
of order $|\alpha|$, on $L^2(\gamma_\infty)$  is defined as
\begin{equation*}
 R^{(\alpha)} = D^\alpha
    \mathcal I_{|\alpha|/2}.
 \end{equation*}
Here
the symbol $\mathcal I_a$ denotes
for any $a>0$
a  Gaussian Riesz potential
given by
\begin{equation}\label{Riesz-potential}
\mathcal I_a f(x)=\frac{1}{\Gamma (a)}
\int_0^{+\infty}
t^{a-1}
\mathcal H_t  \big( P_0^{\perp} f \big)(x)\,dt, \qquad f\in L^2(\gamma_\infty).
\end{equation}

Formally, $\mathcal I_a$ corresponds to the negative power $\big( -\mathcal L\big)^{-a}
\,P_0^{\perp}$.
 In fact, the definition of $\big( -\mathcal L\big)^{-a} $ is the key point in order to define
 $R^{(\alpha)}$, since
 $\mathcal L$ is not self-adjoint in our general framework.
Therefore, we shall now introduce in another way the Gaussian Riesz potentials, and
prove the equivalence  with \eqref{Riesz-potential} for $a > 0$.

In this section, we let $L^2(\gamma_\infty)$ consist of complex-valued funtions.

We first recall from \cite{MPP} that the spectrum of  $-\mathcal L$
is given by
\begin{equation}\label{spectrum_L}
 \sigma (-\mathcal L)=
 \left\{
 \lambda=-\sum_{j=1}^r n_j \lambda_j\,:\, n_j\in\mathbb N    \right\},
\end{equation}
where $ \lambda_1, \ldots,  \lambda_r$ are the  eigenvalues of  the drift matrix $B$.
 In particular,
 $0$ is an  eigenvalue, and the corresponding eigenspace $\ker  \mathcal L$ is one-dimensional and consists of all  constant functions,
 as proved in \cite[Section 3]{MPP}.
Any other point in the spectrum of  $-\mathcal L$ belongs to a fixed cone $\{z \in \C : |\arg z| < \mu\}$
with $\mu < \pi/2$, since the same is true for the numbers $ -\lambda_1, \ldots,  -\lambda_r$.

We also recall that, given a linear operator $L$ on some $L^2$ space, a number $\lambda \in \C$
is a  generalized eigenvalue  of $L$ if there exists a nonzero $u \in L^2$ such that
$(L - \lambda I)^k \,u=0$
for some positive integer $k$. Then $u$
 is called a generalized eigenfunction, and those $u$ span the  generalized eigenspace corresponding  to $\lambda$.

It is known from \cite[Section 3]{MPP} that
the  Ornstein--Uhlenbeck operator $\mathcal L$
 admits a complete system of generalized eigenfunctions, that is, the linear span
 of the generalized eigenfunctions  is dense in $L^2(\gamma_\infty)$.
 Analogous  $L^p$ results are obtained in \cite[Theorem~3.1]{MPP} but will not be used in our paper.

As will be shown in a forthcoming note by the authors \cite{CCS3},   each   generalized eigenfunction of $-\mathcal L$ with a nonzero
eigenvalue is
 orthogonal to the space of constant functions, that is, to the kernel of $-\mathcal L$.
Thus the orthogonal complement of $\ker \mathcal L$ in $L^2(\gamma_\infty)$ coincides with the closure of the subspace generated by all generalized eigenfunctions
with a nonzero generalized eigenvalue. We denote this subspace by $L^2_0(\gamma_\infty)$, so that
 $P_0^{\perp}$  is
the orthogonal projection  onto $L^2_0(\gamma_\infty)$.

The restriction of $-\mathcal L$ to the  generalized
eigenspace corresponding to an eigenvalue $\lambda \ne 0$ has the form
\begin{equation}\label{nilpotent}
 -\mathcal L=\lambda(I+N),
\end{equation}
where $N$ is a nilpotent operator.

Then for $a\in\mathbb R$ one would like to write the power  $(-\mathcal L)^{-a}$, restricted to the
 generalized eigenspace, as
\begin{equation}\label{Taylor-expansion}
  \lambda^{-a}\, (I+N)^{-a} =  \lambda^{-a}\, \sum_{j\ge 0}\frac{ (-a)_j}{j!}\, N^j,
\end{equation}
where the sum is finite and we use the Pochhammer symbol. But here
  $\lambda\in\mathbb C$, and for noninteger $a$ it is not obvious how to choose the value of $\lambda^{-a}$.
For $a>0$ this can be done as follows. We define the argument  $\arg \lambda$ to be in $(-\pi/2,\pi/2)$.

 \begin{proposition}\label{restriction}
Let $\lambda \ne 0$ be a generalized eigenvalue of $-\mathcal L$, and let $a>0$.
   Then the restriction to the corresponding  generalized eigenspaces  of $(-\mathcal L)^{-a}$ defined
   by \eqref{Riesz-potential} is given by \eqref{Taylor-expansion}, where
   $\lambda^{-a} = |\lambda|^{-a}\, e^{-ia\arg \lambda}$.
    \end{proposition}

 \begin{proof}
  The restriction mentioned is

\begin{align}
\frac{1}{\Gamma (a)}
\int_0^{+\infty}
t^{a-1}
\mathcal H_t  \,dt
&=\frac{1}{\Gamma (a)}
\int_0^{+\infty}
t^{a-1}
e^{-t\lambda (I+N)}  \,dt
\\
&=\frac{1}{\Gamma (a)}
\int_0^{+\infty}
t^{a-1} e^{-t\lambda }
 \sum_{j\ge 0}\frac{ (-t\lambda )^j }{j!}\, N^j   \,dt;
 \end{align}
  the sum is again finite and the integral converges.

 Here we make a complex change of variables, letting $\tau = t\lambda$. We arrive at a complex integral
  \begin{align}
\frac{1}{\Gamma (a)}\, \lambda^{-a} \int_{R_\lambda} \tau^{a-1} e^{-\tau }
 \sum_{j\ge 0}\frac{ (-\tau )^j }{j!}\, N^j   \,d\tau,
  \end{align}
 where $R_\lambda$ is the ray $t \, e^{i\arg \lambda}$,  $t\in  \R_+$,  going from $0$ to $\infty$;
 also $\tau^{a-1} = |\tau|^{a-1}\, e^{i(a-1) \arg \lambda}$
  and $\lambda^{-a} = |\lambda|^{-a}\, e^{-ia \arg \lambda}$.
   It is not hard to see that we can move the integration to the positive real axis, getting
  \begin{align*}
 \frac{\lambda^{-a}}{\Gamma (a)}
 \sum_{j\ge 0}\frac{ (-1 )^j }{j!}   \,\int_0^\infty \tau^{j+a-1}\, e^{-\tau }\,d\tau\, N^j & =
 \lambda^{-a}  \sum_{j\ge 0}\,\frac{ (-1 )^j }{j!}\, \frac{\Gamma (j+a)}{\Gamma (a)}\, N^j\\
 & = \lambda^{-a}  \sum_{j\ge 0} \,\frac{ (-a )_j }{j!}\, N^j\,,
  \end{align*}
 since $ (-1)^j \, \Gamma (a+j)  /\Gamma (a) =(-a)_j.  $
   This proves the proposition.
 \end{proof}

In \cite[Lemma 2.2]{Mauceri-Noselli}
it is proved that for each complex number
$a$ such that $\Re a > 0$ the Riesz potential  $(-\mathcal L)^{-a}P_0^{\perp}$
is bounded on $L^2_0(\gamma_\infty)$. Thus  $(-\mathcal L)^{-a}P_0^{\perp}$  is entirely determined
by its restrictions to the generalized eigenspaces, given by
 \eqref{Taylor-expansion}.
  To summarize, this means that by using these restrictions and taking a limit, we get
 a definition of  $(-\mathcal L)^{-a}$ for $a > 0$, which is equivalent to  \eqref{Riesz-potential}.

\medskip

 Finally, let us  comment on the fact that $-\mathcal L$ has real coefficients, although  with $a>0$ \eqref{Taylor-expansion} is not real for nonreal $\lambda$. But if  $(-\mathcal L - \lambda I)^k \,u=0$, then  $\mathcal (-\mathcal L - \bar \lambda I)^k \,\bar u=0$. Thus for  nonreal $\lambda$, the generalized eigenspaces come in
  isomorphic pairs, and the isomorphism is complex conjugation. Conjugating the relation
  $-\mathcal L u = \lambda (I+N) u$, we see that the restriction to the conjugate generalized eigenspace
  is given by $\bar \lambda (I+\bar N)$. The restriction of  $(-\mathcal L)^{-a}$ to the sum of the two conjugate generalized eigenspaces is a real operator. Indeed, this sum is spanned by the functions
  $\Re u$  and $\Im u$, with $u$ in the generalized eigenspace with eigenvalue  $\lambda$.
  For these real functions, we can use the expression  \eqref{Taylor-expansion} for  $(-\mathcal L)^{-a}$
  with $a>0$, since
  \begin{align*}
     (-\mathcal L)^{-a} \,\Re u &= \frac 12\, \left[ (-\mathcal L)^{-a}\, u +  (-\mathcal L)^{-a} \,\bar u\right]\\
      &=    \frac 12  \left[\lambda^{-a} \sum_{j\ge 0}\frac{ (-a)_j}{j!} N^j u +
     \bar  \lambda^{-a} \sum_{j\ge 0}\frac{ ({-a})_j}{j!} \bar N^j\, \bar u\right]
     = \Re \left[\lambda^{-a} \sum_{j\ge 0}\frac{ (-a)_j}{j!} N^j \,u\right],
  \end{align*}
  and similarly for $\Im u ={-i}(u - \bar u)/2$. Thus we see that  $(-\mathcal L)^{-a}P_0^{\perp}$ is a real operator.

 \section{Riesz transforms }\label{Riesz transforms}

We start this section with some technical lemmata.
\begin{lemma}\label{derivative-kernel}
For $i,j\in\{1,\ldots,n\}$, \hskip3pt $x,\,u\in\R^n$
and $t>0$,  one has
\begin{align}
\partial_{x_j} K_t (x,u) &=
 K_t (x,u)\  P_j(t,x,u),\qquad{\text{where}}
 \label{der-1-Kt}
 \\
 P_j(t,x,u)&=
\langle
Q_\infty^{-1}x,\,e_j\rangle
+
{
\left\langle   Q_t^{-1}\,
e^{tB}\, e_j
  \,,\, u-D_t\, x\right\rangle}; \notag\\
\partial_{u_j} K_t (x,u)
&=
 - K_t (x,u)\,{\left\langle Q_t^{-1} e^{tB} \, ( D_{-t}\, u - x), \,e_j\right\rangle};
\label{der-1u-Kt}\\
\partial^2_{x_i  x_j } K_t (x,u) &=
K_t (x,u)\,\left( P_i (t,x,u)  P_j(t,x,u)+
\Delta_{ij} (t)\right),   \qquad{\text{where}}
\label{der-2-Kt}\\
 \Delta_{i j} (t)&= \Delta_{ ji} (t)=
 \partial_{x_i}  P_j(t,x,u) = -\langle
 e_j ,  e^{tB^*}
Q_t^{-1}\, e^{tB}\,
e_i\rangle
;
\notag
\\
\partial_{u_i} P_j(t,x,u)& = \langle   Q_t^{-1}\, e^{tB}\, e_j,\, e_i \rangle.
\label{der-1u-Pj}
\end{align}
\end{lemma}
\begin{proof}
A direct computation, using \eqref{defKR} and \eqref{def:Dtx},
shows that
\begin{align*}
\partial_{x_j} K_t (x,u) &=
 K_t (x,u) \left[
 \langle Q_\infty^{-1}x, e_j\rangle
+
{
\left\langle (
Q_t^{-1}-Q_\infty^{-1}) \,   D_t\, e_j \,,\, u-D_t \:x\right\rangle}\right]
\\
&= K_t (x,u) \left[ \langle
Q_\infty^{-1}x,e_j\rangle
+
{
\left\langle   Q_t^{-1}
e^{tB}\, e_j
  \,,\, u-D_t\, x\right\rangle}\right],\notag
\end{align*}
  yielding
 \eqref{der-1-Kt}.
 An analogous argument leads to  \eqref{der-1u-Kt}.
Rewriting $P_j$
 by means of
\eqref{Dtxii}, one obtains
\begin{align*}
P_j(t,x,u)
= \langle
 e_j
  \,,\, e^{tB^*} Q_t^{-1} \,\big(
  u
-  e^{tB}
  x  \big)
   \rangle,
   \end{align*}
   which implies   \eqref{der-2-Kt} and \eqref{der-1u-Pj}.
\end{proof}
The following lemma provides  a different expression for $P_j$.
\begin{lemma}\label{various-ways-Pj}
One has
\begin{equation*}
P_j(t,x,u) =
\left\langle
 e_j
  ,  e^{tB^*} \ Q_t^{-1}
e^{tB}   \big( D_{-t} \,  u-x\big)\right\rangle
+
\left\langle
 e_j
  ,
Q_\infty^{-1}
D_{-t} \,  u\right\rangle  .
\end{equation*}
\end{lemma}

\begin{proof}
From \eqref{Dtxii} and the expression for $P_j$ in \eqref{der-1-Kt},
 we get
 \begin{align*}
P_j(t,x,u)
&= \langle
Q_\infty^{-1}x,e_j\rangle
+
{
\left\langle   Q_t^{-1}
e^{tB} e_j
  ,D_t \big( D_{-t} \,  u-x\big)\right\rangle}\\
&= \langle
Q_\infty^{-1}x,e_j\rangle
+
{
\left\langle   Q_t^{-1}
e^{tB} e_j
  ,
  \big(  e^{tB} + Q_t e^{-tB^*}Q_\infty^{-1} \big)
  \big( D_{-t} \,  u-x\big)\right\rangle}\\
&= \langle
Q_\infty^{-1}x,e_j\rangle
+
\left\langle   Q_t^{-1}
e^{tB} e_j
  ,
e^{tB}   \big( D_{-t}\,   u-x\big)\right\rangle
+
\left\langle
 e_j
  ,
Q_\infty^{-1}
  \big( D_{-t}\,   u-x\big)\right\rangle
  \\
&=\left\langle
 e_j
  ,  e^{tB^*} \, Q_t^{-1}
e^{tB}   \big( D_{-t} \,  u-x\big)\right\rangle
+
\left\langle
 e_j
  ,
Q_\infty^{-1}
D_{-t} \,  u\right\rangle.
  \end{align*}\end{proof}
As a consequence  of \eqref{der-1-Kt},  Lemma \ref{stimadet} and
 Lemma \ref{various-ways-Pj}, one has for all $j\in \{1, \ldots,n\}$
\begin{equation}
{ |P_j (t,x,u)|\lesssim}
 \begin{cases}
|x|
+ {
| u-D_t \,x
|}/{t} &\text{ if $0 <t \le 1$}, \label{est-for-Pj}\\
e^{-ct} |
   D_{-t} \,  u-x| + |D_{-t}  \, u
| &\text{ if  $t\ge 1$}.
\end{cases}
\end{equation}
Moreover, \begin{equation}\label{est-for-Delta}
\big|\Delta_{ij} (t)
\big|
\lesssim {\big(\min(1,t)\big)}^{-1}\,{e^{-ct}},  \qquad \text{ $t>0$.}
\end{equation}

With
$i,j\in\{1,\ldots,n \}$ we  define  the kernels
\begin{equation*}
\nucleorieszfirstorder (x,u)=\frac{1}{\sqrt \pi}\int_0^{+\infty} t^{-1/2} \ \partial_{x_j} K_t (x,u) \ dt
\end{equation*}
and
\begin{equation*}
\nucleorieszsecondorder (x,u)=
 \int_0^\infty
{\partial^2_{x_i x_j}}
 K_t
(x,u)\,
dt.
\end{equation*}
These integrals are absolutely convergent for all $u \ne x$,
as seen from \eqref{litet}, \eqref{est-for-Pj},  \eqref{est-for-Delta}
and Lemma \ref{expsB-bounded}.
 In  order  to distinguish between small and large values of $t$,
  we split the integrals as
\begin{align*}
\mathcal{R}_{j} (x,u)
=\frac{1}{\sqrt \pi}\left(
\int_0^{1}+\int_1^\infty\right)
t^{-1/2} K_t (x,u) P_j(t,x,u)  \,   dt
\,=:\,  \nucleorieszfirstorderloc (x,u)+ \nucleorieszfirstorderinf (x,u),
   \end{align*}
and
\begin{align*}
\nucleorieszsecondorder  (x,u)
&=\left(
\int_0^{1}+\int_1^\infty\right)
K_t (x,u)\, \left( P_i (t,x,u)  P_j(t,x,u)+
\Delta_{ij} (t)\right)  \,  dt\notag\\
&=: \nucleorieszsecondorderloc (x,u)+\nucleorieszsecondorderinf (x,u).
   \end{align*}
The proof of the next proposition is straightforward  and so omitted.
\begin{proposition}\label{th-nuclei-formale}
The off-diagonal kernels of $R_j$ and $R_{ij}$
are
$ \nucleorieszfirstorder $
and
$\nucleorieszsecondorder$,
in the sense that for $f\in \mathcal C_0^\infty (\R^n)$ and $x\not\in \text{supp} \ f$
\[R_j f(x)=
\int  \nucleorieszfirstorder  (x,u)  f(u)\, d\gamma_\infty (u)\]
and
\[R_{ij}
f(x)=
\int  \nucleorieszsecondorder  (x,u)  f(u)\, d\gamma_\infty (u),\]
where $i,j\in\{1,\ldots,n\}$.
\end{proposition}
The following estimates for
$ \nucleorieszfirstorderloc$ and
$ \nucleorieszsecondorderloc$
 result  from \eqref{litet}, \eqref{est-for-Pj} and  \eqref{est-for-Delta}
\begin{align}\big|\nucleorieszfirstorderloc (x,u)\big|
&\lesssim  e^{R( x)}
\int_0^{1}
t^{-(n+1)/2}  \exp\left(-c\,\frac{|u- D_t\, x |^2}t\right)  \left(|x|
+
\frac1{\sqrt t}
\right)   \,    dt, \label{est-only-kernel-11}
\\
\big|\nucleorieszsecondorderloc (x,u)\big|
&\lesssim   e^{R( x)}
\int_0^{1}
t^{-n/2}   \exp\left(-c\,\frac{|u- D_t\, x |^2}t\right)
\left(|x|^2
 +\frac1t
\right)   \,   dt,
\label{altra-stima-tsmall-order2}
\end{align}
 for all $(x,u)\in \R^n\times\R^n$ with $x\neq u$.

\section{Some estimates for large $t$}\label{Some estimates for large t}
In this section, we  derive some estimates for  $\nucleorieszfirstorderinf$
and $\nucleorieszsecondorderinf$, after some preparations.
\begin{lemma}\label{preliminary-estimates-t-large-not-kernel}
For ${\sigma} \in\{1,2,3\}$ and $x,u\in\R^n$, one has
\begin{equation}
 \int_1^{+\infty}\exp \Big({-\frac14\left|  D_{-t}\,u- x
\right|_Q^2 }\Big)\big|   D_{-t} \,  u\big|^\sigma\, dt \lesssim
1+|x|^{\sigma-1}.\label{stima-5-inf}
\end{equation}
\end{lemma}
\begin{proof}
We can clearly assume that $u \ne 0$. Consider first the case when
 $R(x)\le 1$. Define $t_0 \in \R$ by $R(D_{-t_0}\,u) =2.$
If $t_0>1$, we split the integral at $t = t_0$.

For $1<t<t_0$, \eqref{vel-4} yields
$$R(  D_{-t}\,u)= R(D_{t_0-t}D_{-t_0}\,  u)\ge R( D_{-t_0}\,u)=2\ge  2{R(x)},$$
whence by \eqref{diff-R}
 \[
\big|
D_{-t}\,u- x\big|_Q \simeq \big|D_{-t}\,u\big|_Q,
\]
and by Lemma \ref{expsB-bounded}
\begin{equation*}
  R(  D_{-t}\,u) \gtrsim e^{c(t_0-t)}.
\end{equation*}
Thus
\begin{align*}
 &
 \int_1^{1\vee t_0}
\exp \Big(
{-\frac14
\big|D_{-t}\,u- x\big|_Q^2 }
\Big)\,  \big|
D_{-t} \,  u
\big|^\sigma
  dt\\
&\lesssim
 \int_1^{1\vee t_0}
\exp \Big(
{-c\,
\big|D_{-t}\,u\big|_Q^2 }
\Big)
  \  dt
\\
&\lesssim
 \int_1^{1\vee t_0}
\exp \Big(
{-c\,
e^{c(t_0-t)} }
\Big)
  \  dt
\lesssim
1.\end{align*}
If $t\ge t_0$,
then  $ R(  D_{-t}\,u) \lesssim e^{c(t_0-t)} $,
again because of Lemma \ref{expsB-bounded}, so that
\begin{align*}
 \int_{1\vee t_0}^\infty
\exp \Big(
{-\frac14
\big|D_{-t}\,u- x\big|_Q^2 }
\Big)\,  \big|
D_{-t} \,  u
\big|^\sigma
  dt\lesssim
 \int_{1\vee t_0}^\infty
e^{c\sigma (t_0-t)}  \  dt
\lesssim
1.
\end{align*}
This yields \eqref{stima-5-inf}  in the case $R(x)\le 1$.

Next, assume $R(x) > 1$.
Then  the integral is split at the point $t_1$ defined by  $R(D_{-t_1}\,u) =R(x)/2.$
For $1<t<t_1$ we write
\begin{align} \label{1t1}
  \exp \Big(
{-\frac14\,
\big|D_{-t}\,u- x\big|_Q^2 }
\Big)\,  \big|
D_{-t} \,  u
\big|^\sigma
\lesssim & \,\exp \Big(
{-\frac14\,
\big|D_{-t}\,u- x\big|_Q^2 }
\Big)\,\left(  \big|
D_{-t} \,  u -x
\big|^\sigma + |x|^\sigma\right) \notag
 \\ \lesssim &\,  \exp \Big(
{-c \,
\big|D_{-t}\,u- x\big|_Q^2 }
\Big)\, (1+ |x|^\sigma).
\end{align}
Here we apply the polar coordinates \eqref{def-coord} with $\beta = R(x)$,
writing $x = \tilde x$ and $u = D_{s_u}\, \tilde u$, where
 $s_u\in\R$ and $ R(\tilde u) =  R(x)$.
Then for  $1<t<t_1$
 $$R(  D_{-t}\,u) = R(D_{t_1-t} D_{-t_1}\, u) > R(D_{-t_1}\, u) = {R(x)}/2=\beta/2,$$
and \cite[Lemma 4.3 {\rm(ii)}]{CCS2} applies, saying that
\begin{align*}
\big|D_{-t}\,u-x\big|=  \big|D_{s_u-t}\,\tilde u-\tilde x\big|
\gtrsim |x|   \, |s_u-t|.
\end{align*}
We conclude from \eqref{1t1} that
\begin{align*}
 \int_1^{1\vee t_1}
\exp \Big(
{-\frac14\,
\big|D_{-t}\,u- x\big|_Q^2 }
\Big)\,  \big|
D_{-t} \,  u
\big|^\sigma
\,  dt
&\lesssim
 \int_\R \exp \left(
- c \, |x|^2 \,   |s_u-t|^2
\right)(1+ |x|^{\sigma})\,dt  \\
&\lesssim  |x|^{\sigma-1}.
\end{align*}
For $t> 1\vee t_1$,
we
 deduce from Lemma \ref{expsB-bounded} that
$$
|D_{-t}\,  u|= |D_{t_1-t}\, D_{-t_1} \, u| \lesssim
e^{c (t_1 - t)}\, |D_{-t_1}\,u|\lesssim
e^{c (t_1 - t)}\,|x|.
$$
Moreover, we have
$R(D_{-t}\,u)\le R(x)/{2}$ which implies $|D_{-t}\,u-x|\simeq |x|$
because of \eqref{diff-R}, so that
\begin{align*}
  \int_{1\vee t_1}^{+\infty}
\exp \Big(
{-\frac14
 \big|
 D_{-t}\,u- x\big|^2 }
 \Big)
\big|
D_{-t} \,  u
\big|^\sigma
  \,  dt
\lesssim
\exp(-c |x|^2)
 \int_{1\vee t_1}^{+\infty}
e^{c\sigma (t_1-t)}
  \,  dt \, |x|^\sigma
\lesssim  1.
\end{align*}
We have proved Lemma \ref{preliminary-estimates-t-large-not-kernel}.
\end{proof}

\begin{proposition}\label{preliminary-estimates-t-large}
Let $\sigma_1 \in\{0,1,2,3\}$ and $\sigma_2 \in\{1,2,3\}$. For all $x,u\in\R^n$
one has
\begin{align*}
\int_1^{+\infty}
 K_t (x,u)\,\big|
   D_{-t} \,  u-x
\big|^{\sigma_1} \, \big|   D_{-t} \,  u
\big|^{\sigma_2}
  dt  \lesssim
e^{R(x)}\,
 (1+|x|^{\sigma_2-1})\,.
\end{align*}
\end{proposition}

\begin{proof}
We can delete the factor $\big| D_{-t} \,  u-x \big|^{\sigma_1}$
in the integrand, by replacing the coefficient $1/2$ of the exponential
factor in $K_t$ by $1/4$. Then
this follows from Lemma~\ref{preliminary-estimates-t-large-not-kernel}.
\end{proof}

\begin{corollary}\label{Corollario-tante-stime-t-large}
For all $x,u\in\R^n$
and for all $i,j,k\in\{1, \ldots, n\}$
the following estimates hold:
\begin{align}
&\int_1^{+\infty}
K_t (x,u)\big|
   P_j (t,x,u) \big| \,  dt\lesssim
 e^{R(x)}   \label{KagainstPj} \\
&\int_1^{+\infty}
K_t (x,u) \big|  P_j (t,x,u)  P_i (t,x,u) \big|
\, dt\lesssim
 e^{R(x)}\,(1+|x|),
\label{stima-due-inf-v} \\
&\int_1^{+\infty}
 K_t (x,u)
\big|  P_j (t,x,u)  P_i (t,x,u)  P_k (t,x,u) \big|\,  dt\lesssim
e^{R(x)}\, (1+|x|^2)
\notag,\\
&\int_1^{+\infty}
 K_t (x,u)
\big|  \Delta_{j k} (t)  \big| \, dt\lesssim
e^{R(x)}
\label{stima-4-inf-v},
\\
&\int_1^{+\infty}
 K_t (x,u)
\big|  P_i (t,x,u)  \Delta_{j k} (t)  \big| \, dt\lesssim
e^{R(x)}\notag.
\end{align}
\end{corollary}
\begin{proof}
It is enough to combine  \eqref{tstort},  Lemma
\ref{preliminary-estimates-t-large-not-kernel}
 and Proposition \ref{preliminary-estimates-t-large}
with \eqref{est-for-Pj}
 and \eqref{est-for-Delta}. The quantities $\left|D_{-t}\,u- x\right|$
in the factors $P_j$ can be
replaced by 1, because of the exponential factor in $K_t$.
\end{proof}
\begin{proposition}\label{limitazione-nucleo-larget}
For all  $(x,u)\in\R^n\times\R^n$ and  all $i,j\in\{1, \ldots, n\}$,
\begin{equation}\label{claimvvvvv}
\big|\nucleorieszfirstorderinf (x,u)\big| \lesssim  e^{R( x)}
\end{equation}
and
 \begin{equation}\label{claimvvvvv2}
\big|\nucleorieszsecondorderinf (x,u)\big| \lesssim   e^{R( x)} \,(1+|x|).
\end{equation}
\end{proposition}
\begin{proof}
The expressions for $\nucleorieszfirstorderinf$ and $\nucleorieszsecondorderinf$
in  Section \ref{Riesz transforms} show that
 \eqref{claimvvvvv} follows from \eqref{KagainstPj} and that
 \eqref{claimvvvvv2}  follows  from    \eqref{stima-due-inf-v} and  \eqref{stima-4-inf-v}.
\end{proof}

\begin{remark}\label{rem2}
  If we use polar coordinates with $\beta < 2R(x)$,
 \cite[Lemma 4.3 {\rm(i)}]{CCS2} will imply that  $\left|
D_{-t}\,u- x
\right|_Q  \gtrsim |\tilde u - \tilde x |$.
Then we can use a small part of the factor $\exp\left(-\frac12 \left|
D_{-t}\,u- x
\right|_Q^2 \right) $ in $K_t(x,u)$ to get an extra factor
 $\exp\left(-c\,
|\tilde u - \tilde x |^2\right)  $
in the right-hand sides of all the estimates in  Lemma \ref{preliminary-estimates-t-large-not-kernel},  Proposition \ref{preliminary-estimates-t-large},  
Corollary  \ref{Corollario-tante-stime-t-large} and Proposition \ref{limitazione-nucleo-larget}.

\end{remark}

\section{Some reductions and simplifications}\label{reduction}
This section is closely similar to Section 5 in  \cite{CCS2}.
\smallskip

When we prove  \eqref{thesis-mixed-Di} and  \eqref{thesis-mixed-Di-2},
it is enough to take
$f \ge 0$ satisfying
 $\|f\|_{L^1( \misgaussk)}=1$.

From now on, we also assume that
 $ \lambda > 2$, since otherwise \eqref{thesis-mixed-Di} and  \eqref{thesis-mixed-Di-2}
are obvious.

The $d\gamma_\infty$ measure of the set of points $x$ satisfying  $R( x)> 2 \log \lambda$ is
\begin{align*}
  \int_{ R(x)>  2 \log \lambda }
\exp (-R(x)  )\, dx\,
\notag
\lesssim ( \log \lambda)^{(n-2)/2}\,\exp ({-  2 \log \lambda })
\lesssim  \frac1\lambda,
 \end{align*}
 so this set can be neglected in  \eqref{thesis-mixed-Di} and  \eqref{thesis-mixed-Di-2}.

   \begin{proposition}\label{inner}
Let $x\in \R^n$ satisfy
   $R(x) < \frac12 \log \lambda$.
   Then for all $u\in \R^n$
 \[
|\nucleorieszfirstorderinf (x,u)|\lesssim { \lambda}
\qquad
\text{ and }
\qquad
 \big|\nucleorieszsecondorderinf (x,u)\big|
\lesssim { \lambda}.
\]
             
             If also $(x,u) \in G_1$, the same estimates hold for
 $\nucleorieszfirstorderloc $
and
$ \nucleorieszsecondorderloc $.
\end{proposition}

\begin{proof}
The first statement
   follows immediately from Proposition
\ref{limitazione-nucleo-larget}.

To deal with  $\nucleorieszfirstorderloc $
and
$ \nucleorieszsecondorderloc $, we recall from   \cite[formula (5.3)]{CCS2}
that
\begin{align}    \label{or}
 t^2 \gtrsim \frac1{(1+|x|)^4} &
 \qquad  \text{or}  \qquad
\frac{|u-D_{t}\, x|^2}t
\gtrsim \frac1{(1+|x|)^2\,t} ,
\end{align}
if  $(x,u) \in G_1$ and $0<t<1$.

 From     \eqref{est-only-kernel-11} and  \eqref{altra-stima-tsmall-order2}
we see that both $\big|\nucleorieszfirstorderloc
(x,u)\big|$ and $\big| \nucleorieszsecondorderloc (x,u)\big|$ can be
estimated by a sum of  expressions of type
\begin{align*}
 e^{R(x)}\, (1+| x|)^p  \int_0^{1} t^{-q}   \exp\left(-c\,\frac{| u-D_t\,x|^2}t\right)
\, dt,
\end{align*}
where $p,\: q \ge 0$.
If here we integrate only over those $t\in (0,1)$ satisfying the first inequality in  \eqref{or}, we get at most
\begin{align*}
 e^{R(x)}\, (1+| x|)^p \int_{c(1+|x|)^{-2}}^{1} t^{-q}
\, dt \lesssim e^{R(x)}\, (1+| x|)^C.
   \end{align*}
for some $C$. For the remaining   $t$, the second inequality of  \eqref{or}
holds, and the corresponding part of the integral is no larger than
\begin{align*}
 e^{R(x)}\, (1+|
 x|)^p
\int_0^{1}
t^{-q}   \exp\Big(-\frac c{(1+|x|)^2\,t}\Big)
\, dt
\lesssim e^{R(x)}\, (1+| x|)^C.
   \end{align*}
Obviously $e^{R(x)}\, (1+| x|)^C \lesssim \lambda$ when
$R(x) < \frac12 \log \lambda$,
and the proposition  is proved.
\end{proof}

As a result of this section, we need only consider points $x$ in the ellipsoidal
annulus
\begin{equation*}
{\mathcal E_\lambda}=\left\{
x \in\R^n:\, \frac12 \, \log \lambda\le
R(x)
\le 2  \log \lambda\,
\right\},
\end{equation*}
when proving   \eqref{thesis-mixed-Di} and  \eqref{thesis-mixed-Di-2},
except for $\nucleorieszfirstorderloc $
and
$ \nucleorieszsecondorderloc $   in the local case.

\section{{The case  of large $t$ }}\label{larget}
\begin{proposition}\label{propo-mixed-glob-enhanced}
For all  nonnegative functions $f\in L^1 (\gamma_\infty)$ such that
$\|f\|_{L^1( \gamma_\infty)}=1$, all $i,\,j\in\{1, \ldots,n\}$ and $\lambda>2$
\begin{equation*}
  \qquad \qquad \gamma_\infty
\left\{x :
\int \nucleorieszfirstorderinf (x,u)
f(u)\, d\gamma_\infty(u)
 > \lambda\right\} \lesssim \frac{1}{\lambda\sqrt{\log \lambda}}, \quad
\end{equation*}
and
\begin{equation*}
\gamma_\infty
\left\{x :
\int \nucleorieszsecondorderinf (x,u)
f(u)\, d\gamma_\infty(u)
 > \lambda\right\} \lesssim \frac{1}{\lambda}.
\end{equation*}
In particular,
the  operators with kernels
$\nucleorieszfirstorderinf$ and $\nucleorieszsecondorderinf$
 are of weak type $(1,1)$ with respect to the invariant measure
 $d\misgaussk
$.
\end{proposition}

Notice here that  the estimate for $\nucleorieszfirstorderinf$ is sharpened
by a logarithmic factor. A similar phenomenon occurs for the related
maximal operator; see \cite{CCS2}.

\medspace

\noindent \textit{Proof.}
Having fixed $\lambda>2$, we  use our polar coordinates with
 $\beta=\log\lambda$ and write $x=D_s\, \tilde x$ and
$u=D_{s_u}\, \tilde u$, where  $\tilde x, \tilde u \in \ellipses_\beta$.
We restrict $x$  to the
 annulus  $\mathcal E_\lambda$, in view of Section \ref{reduction}.
It is easily seen that this restriction is possible also with the logarithmic
factor in the case of $\nucleorieszfirstorderinf$.
Applying  the estimates \eqref{claimvvvvv}   and  \eqref{claimvvvvv2}, we insert a factor
 $\exp\left(-c\,
|\tilde u - \tilde x |^2\right)  $,
which is possible because  of Remark \ref{rem2}.
We also replace the factor $1+|x|$ in  \eqref{claimvvvvv2}
by $\sqrt{\log \lambda}$\,.

\medskip

With  $\sigma\in\{1,2\}$
we thus  need to control
 the measure of
the set
 \[
\mathcal A_\sigma (\lambda)=\left\{ x = D_s\,\tilde x \in\mathcal E_\lambda:\:
e^{R(x)}
\int
\exp
 \big(
- c\, \big|
\tilde x-\tilde u\big|^2\big)\,
(\log \lambda)^{(\sigma-1)/2}
\,f(u)\,\misgausskd
 (u)
>
\lambda
\right\}.
\]
The following lemma ends the proof of  Proposition \ref{propo-mixed-glob-enhanced}.
\begin{lemma}\label{bounding-GMAa}
Let
$\sigma\in\{1,2\}$.
The Gaussian measure of
$\mathcal A_\sigma (\lambda)$
satisfies  for $\lambda>2$
\begin{equation}\label{old-prop61}
\gamma_\infty ( \mathcal A_\sigma (\lambda))\lesssim \frac1{\lambda
 \left({\log \lambda }\right)^{(2-\sigma)/2}}.
\end{equation}
\end{lemma}

\begin{proof}
For $\sigma=1$,  \eqref{old-prop61} has been proved in
 \cite[Proposition 6.1]{CCS2}, so we  assume $\sigma=2$.
In view of
\eqref{vel-4}, the function
\[ s\mapsto \exp  \big( R(D_s\, \tilde x) \big)\,
\sqrt{\log \lambda}
\int \exp  \big( - c\, \big| \tilde x-\tilde u\big|^2 \big)
\,f(u)\,\misgausskd (u).\]
 is strictly
 increasing in $s$.
We conclude that  the inequality
\begin{equation}\label{equality-s-alfa}
\exp  \big( R(D_s\, \tilde x) \big)\,\sqrt{\log \lambda} \int \exp \big(- c\, \big| \tilde x-\tilde u\big|^2 \big)\, f(u)\,\misgausskd (u)\, >\, \lambda
\end{equation}
holds
if and only if $s> s_\lambda (\tilde x)$
for some  function
$\tilde x\mapsto s_\lambda (\tilde x) \le \infty$,
with equality for $s=s_\lambda (\tilde x) < \infty$.
Notice also that if the point $x = D_s\, \tilde x $ is in $\mathcal A_2 (\lambda)$ and
thus in $\mathcal E_\lambda$, then
 $|s|<C$ because of Lemma \ref{expsB-bounded}.

We use \eqref{def:leb-meas-pulita} to estimate the
$d\gamma_\infty$ measure of $\mathcal A_2 (\lambda)$.
Since $s$ stays bounded  and
$ |\tilde x|\simeq \sqrt{\log\lambda}$, we obtain
\begin{align*}
\gamma_\infty (\mathcal A_2(\lambda))
&=\int_{\mathcal A_2(\lambda)} e^{-R(x)}\, dx
\\
&\lesssim  { \sqrt{\log \lambda}} \int_{E_{\log\lambda}}
\int_{\substack{s>s_\lambda (\tilde x)\\|s|<C}}
 e^{-R(D_s\, \tilde x)} \,ds\, dS_{\log\lambda}(\tilde x)\\
&\lesssim { \sqrt{\log \lambda}} \int_{E_{\log\lambda}} \int_{s_\lambda (\tilde x)}^{+\infty}
\exp \big( -{R( D_{s_\lambda (\tilde x)}\, \tilde x)-c(s-s_\lambda (\tilde x)})\log\lambda \big) \,ds\, dS_{\log\lambda}(\tilde x) ,
\end{align*}
where the last inequality follows from \eqref{vel-4},
because
$|D_s\, \tilde x|^2 \simeq \log\lambda$
 for $|s|< C$.  Now integrate in $s$, to get
\begin{align*}
\gamma_\infty (\mathcal A_2 (\lambda))
&\lesssim \frac{1}{\sqrt{\log \lambda}}\int_{E_{\log\lambda}}
\exp \big(-R( D_{{s_\lambda  (\tilde x)}}\, \tilde x) \big)\,dS_{\log\lambda}(\tilde x) .
\end{align*}
We combine this estimate
with the case of equality in  \eqref{equality-s-alfa} and change  the order of integration,  concluding that
\begin{align*}
\gamma_\infty  (\mathcal A_2 (\lambda))
\lesssim    \frac{1}{\lambda }
\int   \int_{E_{\log\lambda}} \exp \big( - c\, \big| \tilde x-\tilde u\big|^2 \big) \, dS_{\log\lambda}(\tilde x)\,f(u)\,\misgausskd (u)
\lesssim  \frac{1}{\lambda} \int  f(u)\,\misgausskd (u),
\end{align*}
which proves
 Lemma  \ref{bounding-GMAa}.
 \end{proof}

\section{{ The local case  }}\label{ The local case  }
In this section we define and estimate the local parts of the Riesz operators of orders $1$ and $2$.

Let $\eta$ be a positive smooth function on $\R^n\times\R^n$, such that
$\eta(x,u)=1$ if $(x,u)\in L_A$ and
$\eta(x,u)=0$ if $(x,u)\notin L_{2A}$, for some $A\ge 1$.
Here $A$ will be determined later, in a way that depends only on
$n$, $Q$ and  $B$.
We can assume moreover that
\begin{equation}\label{hypo-eta}
\big| \nabla_x \,\eta (x,u)\big|+
\big| \nabla_u\, \eta (x,u)\big|\lesssim
|x-u|^{-1}, \qquad x\neq u.
\end{equation}
We  introduce the global and   local parts of the first-order Riesz
transform $R_j$ by
\begin{equation*}
R_{j}^{{\rm{glob}}} f(x)=\int
\mathcal{R}_{j} (x,u)
\big(1-\eta(x,u)\big) f(u) \,d\gamma_\infty (u), \qquad
f\in \mathcal C^{\infty}_0(\R^n),
\end{equation*}
 and
\begin{equation*}
R_{j}^{{\rm{loc}}} = R_j - R_{j}^{{\rm{glob}}}.
\end{equation*}
The off-diagonal kernel of $R_{j}^{{\rm{loc}}}$ is
 $\mathcal R_{j}(x,u)\eta(x,u)$.
  For the second-order Riesz transforms, we simply repeat the above with the subscript $j$ replaced by $ij$.

\medskip
To prove the weak type $(1,1)$ of the operators $R_{j}^{{\rm{loc}}}$ and
$R_{ij}^{{\rm{loc}}}$,
 we shall verify that  their kernels $\mathcal{R}_{j} \eta$
and  $\mathcal{R}_{ij} \eta$
satisfy the standard Calder\'on--Zygmund
estimates.

We first need a  lemma.

\begin{lemma}\label{lemma-trivial-integral}
Let $p,\: r\ge 0$ with $p+ r/2 > 1 $, and $(x,u)\in L_{2A}$ with $x\neq u$. Then for $\delta>0$
\begin{equation*}
\int_0^{1}
t^{-p}   \exp\left(-\delta\,\frac{|u- D_t \,x |^2}t\right) |x|^{r} \,  dt\lesssim {C(\delta, p,r)}\,{|u-x|^{-2p-r+2}},
\end{equation*}
where the constant $C(\delta, p,r)$ may also depend on  $n$, $Q$,  $B$ and $A$.
\end{lemma}

\begin{proof} Write
  \begin{equation*}
  |u- D_t\, x |^2 = |u-x|^2 +  |x- D_t\, x|^2 + 2\langle u-x,  x- D_t\, x \rangle.
  \end{equation*}

Since $|x- D_t \,x| \simeq t |x|$ and  $(x,u)\in L_{2A}$, the absolute value of the  last term here
is no larger than $CAt$. It follows that
\begin{equation*}
  |u- D_t \,x |^2/t \ge |u-x|^2/t + c  t |x|^2 - CA.
\end{equation*}
We now apply this to the integral in the lemma, and estimate
$\exp\left( -c\delta  t |x|^2\right)$ by $C\delta^{-r/2} t^{-r/2}\, |x|^{-r}$.
The integral is thus controlled by
\begin{equation*}
\delta^{-r/2}  \int_0^{1} t^{-p-r/2}  \exp\left(-\delta\,\frac{|u- x |^2}t\right) \,  dt,
\end{equation*}
and the required estimate follows
via the change of variables $s= {|u-  x |^2}/t$.
\end{proof}

\begin{proposition}\label{lemma-Calderon}
Let the function $\eta$ be as above.
For all $(x,u)\in L_{2A}$, $x\neq u$,
and all $j\in\{1, \ldots, n\}$,
the following estimates hold:
\begin{enumerate}
\item[{\rm{(i)}}]  $\big|
\mathcal{R}_{j} (x,u)
\
\eta (x,u)\big| \lesssim   e^{R( x)}\,{|u-x|^{-n}}$;
\medskip

\item[{\rm{(ii)}}]
$\big|\nabla_x \big(
\mathcal{R}_{j} (x,u)
\
\eta (x,u)\big)\big|
 \lesssim {  e^{R( x)}\,}{|u-x|^{-(n+1)}}$;
\medskip

\item[{\rm{(iii)}}]
$\big|\nabla_u \big(
\mathcal{R}_{j} (x,u)
\
\eta (x,u)\big)\big|
 \lesssim {  e^{R( x)}\,}{|u-x|^{-(n+1)}}$,
\end{enumerate}
with implicit constants depending on $n$, $B$, $Q$ and $A$.\\
The same estimates hold for  $\mathcal{R}_{ij},\;\; i,\,j\in\{1, \ldots, n\}$.
\end{proposition}

\begin{proof}
We start with $\mathcal{R}_{j}$.

\noindent {\it{(1)}}
From \eqref{est-only-kernel-11}
we obtain
\begin{align*}
\big|\nucleorieszfirstorderloc (x,u)\big|
&\lesssim e^{R( x)}
\int_0^{1}
t^{-(n+1)/2} \  \exp\left(-c\,\frac{|u- D_t\, x |^2}t\right)   \left(|x|
+
\frac1{\sqrt t}
\right)  \    dt \notag\\
&\lesssim {  e^{R( x)}\,
}{|u-x|^{-n}},
\end{align*}
 by Lemma \ref{lemma-trivial-integral}.
Further,
\eqref{claimvvvvv} implies the desired estimate for
$\nucleorieszfirstorderinf$,
and {\rm{(i)}} follows.
\medskip

\noindent {\it{(2)}}  As a consequence of \eqref{hypo-eta},
one has for $x\neq u$
\begin{align*}
\big|
\partial_{x_\ell}
\big(
\mathcal{R}_{j} (x,u)
\
\eta (x,u)\big)\big|
&\lesssim
\big|
 \partial_{x_\ell}\, \mathcal{R}_{j} (x,u) \big|
+  |x-u|^{-1} \big| \mathcal{R}_{j} (x,u) \big|.
 \end{align*}
Since item {\it(1)} above takes care of the last term here,
it suffices to show  that
\begin{equation}\label{est-only-for-derivative}
\big|
\partial_{x_\ell}\,
\mathcal{R}_{j} (x,u)
\big|
\lesssim  { e^{R( x)} }\,{|u-x|^{-(n+1)}}.
\end{equation}
For $\mathcal{R}_{j,0}$ we get from
 \eqref{der-1-Kt}, \eqref{der-2-Kt},  \eqref{est-for-Pj}, \eqref{est-for-Delta}, combined with \eqref{litet}
\begin{multline*}
\left|
\partial_{x_\ell}\,
 \nucleorieszfirstorderloc(x,u)\right|
\lesssim
\int_0^{1}
t^{-1/2} K_t (x,u)
\left|
  P_\ell(t,x,u)   \, P_j(t,x,u) +
\Delta_{j\ell} (t) \right| \,   dt \\
\lesssim
\,e^{R( x)} \int_0^{1}
t^{-(n+1)/2} \,
 \exp\left(-\frac12\,\frac{|u-D_t \,x |^2}t\right)
\left[
 \left( |x|
+
\frac{
| u-D_t \,x
|}{t}\right)^2
 +\frac1t\
 \right] \,   dt.
 \end{multline*}
In the last
factor here, we can replace
$|u-D_t \,x |$ by  $\sqrt t$,
  reducing slightly the factor $1/2$
in the exponential expression.
This will be done repeatedly   
 in the sequel. We arrive at
\begin{align*}
\big|\partial_{x_\ell}\,
 \nucleorieszfirstorderloc(x,u)\big|
\lesssim
 e^{R( x)} \int_0^{1}
t^{-(n+1)/2} \,
 \exp\left(-\frac14\,\frac{|u-D_t\, x |^2}t\right)
 \left( |x|^2
+
\frac1{ t}\right)
 \,   dt,
 \end{align*}
and
 Lemma \ref{lemma-trivial-integral}
allows us to estimate this by $ { e^{R( x)} }\,{|u-x|^{-(n+1)}}$ as desired.

For $ \nucleorieszfirstorderinf$ \eqref{stima-due-inf-v} and \eqref{stima-4-inf-v} imply that
\begin{align*}
\big|
\partial_{x_\ell}\,
 \nucleorieszfirstorderinf(x,u)\big|
\lesssim
\int_1^{+\infty}
K_t (x,u)
\left|P_\ell(t,x,u)   \, P_j(t,x,u) +
\Delta_{j\ell} (t)
 \right| \,   dt
\lesssim
e^{R( x)}\,
(1+|x|).
 \end{align*}
Here
$
1+|x|
 \lesssim |x-u|^{-1} \lesssim  |x-u|^{-(n+1)}$, and
 \eqref{est-only-for-derivative} is verified, proving {\rm{(ii)}} as well.

\medskip
\noindent {\it{(3)}}
As in item {\it{(2)}},
it suffices to estimate $\left| \partial_{u_\ell}\, \mathcal{R}_{j} (x,u) \right|$.
Because of
 \eqref{der-1u-Kt}, \eqref{der-1u-Pj}             
 and
\eqref{est-for-Pj}, we have
\begin{align*}
\big|
\partial_{u_\ell}\,
 \nucleorieszfirstorderloc&(x,u)\big|\\
&\lesssim
\int_0^{1}
t^{-1/2}  K_t (x,u)
 \left|
     P_j(t,x,u)  \,
\left\langle Q_t^{-1}e^{tB} \, ( D_{-t}\,u- x), e_\ell \right\rangle
+
{
\left\langle  Q_t^{-1}e^{tB} e_j, e_\ell\right\rangle}
 \right| \,   dt \\
&\lesssim
\int_0^{1}
t^{-1/2}
 K_t (x,u)
\left[
 \left( |x|
+
\frac{
| u-D_t\, x
|}{t}\right)
\,   \frac{
| D_{-t}\,u- x
|}{t}
 +\frac1t\,
 \right]\,   dt \\
&\lesssim e^{R( x)}\,
\int_0^{1}
t^{-(n+1)/2} \
  \exp\left(-c\,\frac{|u-D_t \, x |^2}t\right)
\left(
\frac{|x|}{\sqrt t}+
\frac{1}{t}
 \right) \,   dt \\
&\lesssim  e^{R( x)}\, {|u- x|^{-(n+1)}},
  \end{align*}
where we proceeded much as in item   {\it{(2)}}.
Similarly,
 \begin{align*}
\big|
\partial_{u_\ell}\, &
 \nucleorieszfirstorderinf (x,u)\big|
\lesssim
\int_1^{\infty}
  K_t (x,u)\, \Big(
     |P_j(t,x,u)|  \,
\left|
Q_t^{-1} e^{tB} \ (D_{-t}\, u- x)\right|+
\left| Q_t^{-1} e^{tB}  e_j
\right|
 \Big) \,   dt \\
&\lesssim
\int_1^{\infty}
  K_t (x,u) \,
\Big[\Big(
 e^{-ct}  \big|
  D_{-t} \,  u-x
\big|+\big|
D_{-t} \,  u
    \big|\Big)
 \,  e^{-ct}   \,
\big|
  D_{-t}\, u- x\big|
+  e^{-ct} \Big]
 \,   dt \\
&\lesssim
\int_1^{\infty}
  K_t (x,u)\,
 e^{-ct}
    \,\left[ \left|   D_{-t} \,  u-x
\right|^2+\left|
D_{-t} \,  u
    \right|
     \,
\left|
  D_{-t}\, u- x\right|
+ 1 \right] \,   dt
\;\lesssim\;
 e^{R( x)},
  \end{align*}
as follows from
Proposition \ref{preliminary-estimates-t-large}.

Items  {\rm{(i)}},  {\rm{(ii)}}  and {\rm{(iii)}} are proved for
$\nucleorieszfirstorder$, and
we now turn to  $\nucleorieszsecondorder$.

\medskip

\noindent {\it{(1')}} For $(x,u) \in L_{2A}$,   it results  from \eqref{altra-stima-tsmall-order2} and Lemma \ref{lemma-trivial-integral}  that
\begin{align}  
\big|\nucleorieszsecondorderloc (x,u)\big|
&\lesssim  e^{R( x)}
\int_0^{1}
t^{-n/2} \,  \exp\left(-c\,\frac{|u- D_t\, x |^2}t\right) \,  \left(|x|^2
 +\frac1t
\right)  \,        dt \notag
\lesssim \frac{  e^{R( x)} }{|u-x|^n}.
\end{align}
Since  \eqref{claimvvvvv2} gives the estimate for $\nucleorieszsecondorderinf$,
 item~{\rm{(i)}} is verified.

\medskip

\noindent {\it{(2')}}
As before, we need only consider the derivative  $ \partial_{x_\ell}\, \nucleorieszsecondorder  (x,u) $
 in the local region.
From  \eqref{der-1-Kt} and \eqref{der-2-Kt}, we have
\begin{align*}
\partial_{x_\ell}\,
 \nucleorieszsecondorder (x,u)
 =
  \int_0^{\infty}
  K_t (x,u) \Big[
  & P_\ell(t,x,u)   \, \Big( P_i (t,x,u)  P_j(t,x,u)+\Delta_{ij} (t)\Big)
\\
&  +\Delta_{{i \ell}} (t)\, P_j(t,x,u) +\Delta_{{j \ell}} (t)\, P_i(t,x,u)  \Big]
  \,   dt.
\end{align*}
For $0<t<1$, we estimate the factors of type $P_i$ and  $\Delta_{ij}$ here
 by means of \eqref{est-for-Pj} and \eqref{est-for-Delta}. Then
we use the exponential factor in $K_t$ to replace $|u-D_t\, x|$ by $\sqrt t$,
and apply   Lemma \ref{lemma-trivial-integral}.
The result will be
\begin{align*}
\big|
\partial_{x_\ell}\,
 \nucleorieszsecondorderloc(x,u)\big|
\lesssim &\; e^{R( x)}
\int_0^{1}
t^{-n/2} \,  \exp\left(-c\,\frac{|u- D_t\, x |^2}t\right)
 \left(
 |x|^3
+\frac1{{t\sqrt t}}
\right)\,    dt
\\
\lesssim &\;  {e^{R( x)}  }\,{|u-x|^{-(n+1)}}.
 \end{align*}

For $t>1$, we use  \eqref{est-for-Pj} and  \eqref{est-for-Delta}, getting
\begin{align*}
&\big|
\partial_{x_\ell}\,
 \nucleorieszsecondorderinf(x,u)\big| \\
 & \lesssim
\int_1^{\infty}
  K_t (x,u)\, \left( e^{-ct}
 \left| D_{-t} \,u- x\right|^3 +\left|D_{-t}\, u\right|^3
+ e^{-ct}\left| D_{-t}\, u- x\right|+ e^{-ct} \left|D_{-t}\, u\right|
 \right) \,   dt \\
&\lesssim
 e^{R( x)} \,|x-u|^{-2},
  \end{align*}
because of  Lemma \ref{preliminary-estimates-t-large}.
\medskip

\noindent {\it{(3')}}
Applying  \eqref{der-1u-Kt} and  \eqref{der-1u-Pj}, we have
   \begin{multline*}
\partial_{u_\ell}\,
 \nucleorieszsecondorder(x,u)\\
=
\int_0^{\infty}
  K_t (x,u)
 \Big[
- \left\langle Q_t^{-1}e^{tB} \ ( D_{-t}\,u- x), e_\ell \right\rangle
\left( P_i (t,x,u)  P_j(t,x,u)+\Delta_{j i} (t)\right)\\
+
\left\langle  Q_t^{-1}e^{tB} e_i, e_\ell\right\rangle P_j(t,x,u)
+ P_i (t,x,u) \left\langle  Q_t^{-1}e^{tB} e_j, e_\ell\right\rangle
 \Big] \,   dt.
\end{multline*}
Arguing as before, we conclude
  \begin{align*}
\left|\partial_{u_\ell}\,
 \nucleorieszsecondorderloc(x,u)\right|& \\
\lesssim
\int_0^{1}
 K_t (x,u) &
  \left[ \frac{
| D_{-t}\,u- x
|}{t}
 \left(
 \left( |x|
+
\frac{
| u-D_t \,x
|}{t}\right)^2
\,
 +\frac1t\right)
+ \frac1t \left( |x|
+
\frac{
| u-D_t\, x
|}{t}\right)
 \right]\,   dt \\
\lesssim  e^{R( x)}\,
\int_0^{1} &
t^{-n/2} \,
  \exp\left(-c\,\frac{|u-D_t\,  x |^2}t\right)
\Big(
\frac{|x|^2}{\sqrt t}+
\frac{1}{t\sqrt t}
 \Big) \,   dt \\
\lesssim  e^{R( x)}\,& {|u-x|^{-(n+1)}}.
  \end{align*}
Further,
 \begin{align*}
\big|
\partial_{u_\ell}\,
 \nucleorieszsecondorderinf&(x,u)\big| \\
  \lesssim
\int_1^{\infty}&
  K_t (x,u)\, e^{-ct} \big(
 \left| D_{-t} \,u- x\right|^3
+\left| D_{-t}\, u- x\right| \left|D_{-t}\, u\right|^2 + 1
+\left| D_{-t}\, u- x\right| + \left|D_{-t}\, u\right|
 \big) \,   dt \\
\lesssim
 \,e^{R( x)} &\, |x-u|^{-1},
  \end{align*}
the last step from Proposition \ref{preliminary-estimates-t-large}.

This completes the proof  of
Proposition \ref{lemma-Calderon}.
\end{proof}

We can now prove the weak type $(1,1)$ boundedness of the local parts of the Riesz transforms.
\begin{proposition}\label{propo-locale}
For $i,\, j\in\{1, \ldots, n\}$,
the operators ${R}^{\rm{loc}}_{j}$
and
${R}^{\rm{loc}}_{ij}$
are of weak type $(1,1)$ with respect to the invariant measure $d\gamma_\infty$.
\end{proposition}
\begin{proof}
This  is now a straightforward  consequence of \cite[Proposition 2.3]{Mauceri-Noselli},
 \cite[Proposition 3.4]{GCMST2}, our Proposition \ref{lemma-Calderon}
and  \cite[Theorem 3.7]{GCMST2}.
\end{proof}

\section{{ The global case for small $t$}}\label{smallt-global}
In this section, we study the operators $R_{j,0}^{\rm{glob}}  $  and
$R_{ij,0}^{\rm{glob}}$,
with kernels
$\big(1-\eta)\nucleorieszfirstorderloc $ and  $\big(1-\eta)\nucleorieszsecondorderloc $, respectively. The function
 $\eta$ was defined in the beginning of Section
\ref{ The local case  }, depending on $A$.
 We have the following  result.

\begin{proposition}\label{propo-glob-tsmall}
For $i,j\in\{1, \ldots, n\}$,
the operators   $R_{j,0}^{\rm{glob}}  $ and $R_{ij,0}^{\rm{glob}}$
are of weak type $(1,1)$ with respect to the invariant measure $d\gamma_\infty$,
provided $A$ is chosen large enough.
\end{proposition}

  The estimates
  \eqref{est-only-kernel-11} and \eqref{altra-stima-tsmall-order2}
show that to prove this proposition,
it suffices  to verify the weak type $(1,1) $ boundedness of the operator with kernel $\int_0^{1} \mathcal K_t (x,u)\,
   dt \, \chi_{G_A} (x,u) $, where
\begin{align*} \mathcal K_t (x,u)
=  e^{R( x)}\,
t^{-n/2} \  \exp\left(-c\,\frac{|u- D_t \,x |^2}t\right) \,
 \left( |x|^2 +
\frac1{ t}
\right).                  
\end{align*}
      As is clear from Section \ref{reduction}, we need only consider the
case $|x|\gtrsim 1.$
 This assumption will be valid in the rest of this section.

The sets
\begin{align*}
 I_m(x,u)
=\left\{t \in(0,1): \:
 2^{m-1} \sqrt t< |u-D_t\, x |\leq
 2^{m} \sqrt t
\,\right\}, \qquad m = 1,2,\dots,
\end{align*}
and
\begin{align*}
 I_0(x,u)
=\left\{t \in(0,1): \:
   |u-D_t\, x |\leq
  \sqrt t
\,\right\}
\end{align*}
together form a partition of  $(0,1)$. For $t \in  I_m(x,u)$,
\begin{align*} \mathcal K_t (x,u)
\le  e^{R( x)}\
t^{-n/2} \,  \exp\left(-c\,2^{2m}\right) \,
 \left( |x|^2 +
\frac1{ t}
\right).                   
\end{align*}
Let
\begin{equation*}
\mathcal Q_m(x,u) =
  e^{R( x)}
 \int_{I_m(x,u)} t^{-n/2}
 \left( |x|^2 +
\frac1{ t}
\right)
dt  \,  \chi_{G_A} (x,u),       \qquad \;
m = 0, 1,\dots.
   \end{equation*}
The operator we need to consider  has kernel
  \begin{equation} \label{suminm}
  \sum_{m=0}^\infty \,  \exp\left(-c\,2^{2m}\right) \,  \mathcal Q_m(x,u).
\end{equation}
\begin{proposition}\label{Rm}
Let $m \in \{0, 1,\dots \}$.
The operator whose kernel is
$\mathcal{Q}_{m} $
 is of weak type $(1,1)$ with respect to $d\gamma_\infty$,
with a quasinorm bounded by
$C\,2^{Cm}$ for some $C$.
\end{proposition}

This proposition implies Proposition \ref{propo-glob-tsmall},
 since the factors $ \exp\left(-c\,2^{2m}\right)$
in \eqref{suminm}
 will allow us to sum over $m$ in the space $L^{1,\infty}(\gamma_\infty)$.
Before proving  Proposition \ref{Rm}, we make some preparations.

From now on, we fix  $m \in  \{0,1,\dots\}$. If  $ t \in I_m(x,u)$,
 Lemma \ref {differ} implies
\begin{equation} \label{dista}
  |u-x| \le |u - D_t\, x| + |D_t\, x - x| \lesssim  2^m\sqrt t +  t|x|,
\end{equation}
   and further
\begin{align}\label{rtxu}
   |  R(D_t\, x) -R(u)| =\,\frac12\, \big(  |D_t \,x|_Q+|u|_Q\big)\,
\big| |D_t \,x|_Q - |u|_Q\big| \notag
&\lesssim \, (|x|+|u|)\:|D_t\, x - u|_Q  \\
&
\lesssim (|x|+|u|)\: 2^m\sqrt t \,.
  \end{align}

\begin{lemma} \label{im}
Let  $(x,u) \in G_A$.
  If $A$ is chosen large enough, depending only on  $n$,   $Q$ and  $B$,
then
\begin{equation*}
  I_m(x,u) \subset (2^{-2m}/|x|^2,\: 1), \qquad m = 0,1, \dots .
\end{equation*}
\end{lemma}

\begin{proof}
If $ t \in I_m(x,u)$ but  $t \le 2^{-2m}/|x|^2$, the   two terms to the right in
\eqref{dista} are both bounded by $1/|x|$, so that
 $|x-u| < C/|x|$ for some $C$. Since we assume $|x|\gtrsim 1$,
this will violate the hypothesis $(x,u) \in G_A$, if $A$ is
large. The lemma follows.
\end{proof}

\begin{lemma}  \label{size}
Let  $ t \in I_m(x,u)$.
If the constant  $C_0 > 4$  is chosen large enough,
depending only on $n$,   $Q$ and  $B$,
then  $t > C_0\,2^{2m}/|x|^2$
 implies
\begin{align}
  |u|  &\simeq |x|,\label{u,x}
\\  
   R(u) -  R(x) &\simeq t|x|^2  \simeq |u-x|\,|x|
    \label{Ru-Rx} \\
 and \hspace{5.5cm}&\hspace{8.8cm} \\
 t &\simeq  {|u-x|}/{|x|}.  \label{u-x}
\end{align}
\end{lemma}

\begin{proof}
Because of our assumptions on $t$,  \eqref{dista} implies that
$|u-x| \lesssim t|x| \lesssim |x|$,
and so $ |u| \lesssim|x| $. This proves
one of the inequalities in both  \eqref{u,x} and \eqref{u-x}.
Aiming at \eqref{Ru-Rx}, we write\begin{equation*}
    R(u) -  R(x) = R(D_t \,x) -  R(x)  -  (R(D_t\, x)- R(u)).
\end{equation*}
From \eqref{stime per R-t(x)-R(x)} it follows that
\begin{equation*}
  R(D_t \,x) -  R(x) \simeq t|x|^2,
\end{equation*}
and \eqref{rtxu} and our assumptions lead to
\begin{align*}
   |  R(D_t\, x) -R(u)|
\lesssim\,|x|\, 2^m\sqrt t
 \lesssim \, t|x|^2/\sqrt{C_0}\,.
\end{align*}

Thus we can choose $C_0>4$ so large that
$ | R(D_t \,x) - R(u)| < (R(D_t\, x) -  R(x))/2$, and the first estimate in
\eqref{Ru-Rx} follows.
In particular, $R(u) > R(x)$, and so $|u|\gtrsim  |x|$, which completes the
proof of \eqref{u,x}.
We also obtain  the remaining part of  \eqref{u-x},
by writing
\begin{equation*}
   t|x|^2 \simeq R(u) -  R(x) = \frac12\, (|u|_Q+  |x|_Q)\,(|u|_Q -  |x|_Q)
\lesssim  |x|\,|u - x|_Q.
\end{equation*}
Finally, the second estimate in  \eqref{Ru-Rx} is a trivial consequence of
 \eqref{u-x}.

The lemma is proved.
\end{proof}

In view of the last two lemmas,
we split  $I_m(x,u)$ into
\begin{equation*}
   I_m^-(x,u) =  I_m(x,u) \cap (2^{-2m}/|x|^2,\:1\wedge  C_0\,2^{2m}/|x|^2 )
\end{equation*}
and
\begin{equation*}       I_m^+(x,u) =  I_m(x,u) \cap (1\wedge  C_0\,2^{2m}/|x|^2,\: 1 ).
\end{equation*}

Define for $ m = 0,1,\dots$ and $|x|\gtrsim 1$
\begin{align*}
\mathcal Q_m^-(x,u) =
  e^{R( x)}\
 \int_{I_m^-(x,u)} t^{-n/2} \,
 \left( |x|^2 +
\frac1{ t}
\right)\, dt  \  \chi_{G_A} (x,u)
   \end{align*}
and
\begin{align*}
\mathcal Q_m^+(x,u) =
  e^{R( x)}\
 \int_{I_m^+(x,u)} t^{-n/2} \,
 \left( |x|^2 +
\frac1{ t}
\right)\, dt  \  \chi_{G_A} (x,u),
   \end{align*}
so that
$\mathcal Q_m(x,u) = \mathcal Q_m^-(x,u) + \mathcal Q_m^+(x,u) $.

\begin{lemma}\label{tipoforte}
The operator with kernel
$\mathcal{Q}^{-}_{m} $
 is of strong type $(1,1)$ with respect to $d\gamma_\infty$, with a norm bounded by
$C\,2^{Cm}$.
\end{lemma}

\begin{proof}
 For $ t < C_0\,2^{2m}/|x|^2$, the estimate \eqref{dista} implies
  \begin{equation} \label{qm-}
  |u-x| \le 2\,C_0\, 2^{2m}/|x|,
  \end{equation}
which leads to
\begin{align*}
 \mathcal Q_m^-(x,u) \lesssim &
 \ e^{R( x)}
 \int_{2^{-2m}/|x|^2}^{C_0\,2^{2m}/|x|^2} t^{-n/2} \,
 \left( |x|^2 +
\frac1{ t}
\right)\, dt  \  \chi_{\left\{|u-x|\lesssim 2\,C_0\, 2^{2m}/|x|  \right\}}\\
 \lesssim & \ e^{R( x)}\,2^{Cm}\,|x|^{n}\,
\chi_{\left\{|u-x|\lesssim  2\,C_0\, 2^{2m}/|x| \right\}},
  \end{align*}
for some $C$.
\medskip

Consider first the case $|x| \le C_0\,2^m$.
Then $ \mathcal Q_m^-(x,u) \lesssim e^{R( x)}\, 2^{Cm}$, and so
\begin{equation*}
 \int_{|x| \le C_0\,2^m}  \mathcal Q_m^-(x,u)\, d\gamma_\infty (x) \lesssim  2^{Cm}.
\end{equation*}
Since this is uniform in $u$, the strong type follows for  $|x| < C_0\,2^m$.

\medskip

To  deal with points $x$ with $|x| > C_0\,2^m$,
we introduce dyadic rings
\begin{equation*}
 L_i = \{ x:  C_0\,2^{m+i} < |x| \le   C_0\,2^{m+i+1}\}, \qquad i = -1, 0, 1,\dots\:.
\end{equation*}
If $x \in L_i$ with $i\ge 0$, it follows from \eqref{qm-}
 that
\begin{equation*}
  |u-x|  < 2^{m-i+1} < C_0\, 2^{m-i-1},
\end{equation*}
the last step since $C_0 > 4$.
The triangle inequality now shows that
$u$ is in the extended ring
\begin{equation*}
  L_i' = L_{i-1} \cup L_i  \cup L_{i+1}.
\end{equation*}

With $0 \le f \in L^1(\gamma_\infty)$ we
let $F(u) = e^{-R( u)}\,f(u)$, so that $\int f\,d\gamma_\infty = \int F\,du $.
Then for  $x \in L_i, \;\: i\ge 0$,
 \begin{align*}
\int
\mathcal{Q}^{-}_{m} (x,u)\,
f(u)\, d\gamma_\infty (u)
&  \lesssim e^{R(x)}\,
{2^{Cm}}\, 2^{n(m+i)}
\int_{|u-x|\lesssim C_0\, 2^{m-i-1}}
  F(u)\, du\\
&= e^{R(x)}\,
{2^{Cm}} \, \Psi * F(x),
\end{align*}
where
 $\Psi$ is given by
\begin{align*}
\Psi(u)=  2^{n(m+i)} \, \chi_{B\left(0,\, C_0 \, 2^{m-i-1} \right)}(u).
\end{align*}
Since $\int \Psi (u)\,du\lesssim 2^{Cm}$,
we can integrate in $x$ to get
\begin{align*}
\int_{L_i}  d\gamma_\infty (x)
\int \mathcal{Q}^{-}_{m} (x,u)\,
f(u)\, d\gamma_\infty (u)
\lesssim 2^{Cm}
\int_{L_i}
 \Psi * F(x) \,  dx
\lesssim 2^{Cm}
\int_{L_i'}
 F(u) \,  du.
 \end{align*}
Summing over $i\ge 0$, we get
\begin{equation*}
  \int_{|x|> C_0 2^m}  d\gamma_\infty (x)
\int \mathcal{Q}^{-}_{m} (x,u)\, f(u)\, d\gamma_\infty (u)
\lesssim  2^{Cm} \, \sum_{i =-1}^\infty \int_{L_i'}\, F(u)\, d\gamma_\infty (u)
\lesssim   2^{Cm} \, \int f\,d\gamma_\infty.
\end{equation*}
The lemma follows.
\end{proof}

\begin{lemma}\label{tipodebole}
The operator with kernel
$\mathcal{Q}^{+}_{m} $
 is of weak type $(1,1)$ with respect to $d\gamma_\infty$, and its quasinorm is bounded by
$C\,2^{Cm}$.
\end{lemma}

\begin{proof}
The support of the kernel $ \mathcal Q_m^+$ is  contained in the set
\begin{equation*}
\cone_m = \left\{(x,u)\in G_A: \;
 I_m^+(x,u) \ne \emptyset
\right\}.
\end{equation*}

 We first sharpen \eqref{u-x}
                     by restricting $t$ further.
Because of \eqref{rtxu}, \eqref{u,x} and \eqref{u-x},
 any $t\in I_m^+(x,u)$ satisfies
\begin{align} \label{restr-t}
\big|R(D_t \,x) -R(u)\big|
\lesssim |x|\,2^m\,\sqrt t
\lesssim  2^m\, \sqrt {{|x-u|}{|x|}},
\end{align}
and from \eqref{vel-4} we know that
\begin{equation*}
  \partial_t R(D_t\, x)
\simeq |x|^2.
\end{equation*}
The size of this derivative shows that \eqref{restr-t} can  hold only
for $t$ in
an interval of length at most $ C\,  2^m\, \sqrt{|x-u|}/ |x|^{3/2}$,
call it $I$. We obtain, using  \eqref{u-x} again,
\begin{align*}
 \mathcal Q_m^+(x,u) \lesssim &
 \ e^{R( x)}\,
 \int_{I\cap I_m^+(x,u)}  \left(\frac{|x-u|}{|x|}\right)^{-n/2}
 \left( |x|^2 +
\frac{|x|}{|x-u|}
\right)\, dt \, \chi_{\cone_m}(x,u).
 \end{align*}
The global condition implies  $ |x|/|x-u| \lesssim |x|^2$, so that
\begin{align*}
 \mathcal Q_m^+(x,u) \lesssim  \ e^{R( x)}\, 2^m\,
|x|^{{(n+1)}/2}\, |x-u|^{{(1-n)}/2} \, \chi_{\cone_m}(x,u)
 =: \mathcal{M}_{m} (x,u).
  \end{align*}
It will be enough to prove Lemma  \ref{tipodebole} with  $\mathcal Q_m^+$
replaced by the kernel $\mathcal{M}_{m}$ thus defined.

With $\lambda >2$ fixed, we assume   $x \in \mathcal E_\lambda$.
We use our polar coordinates       with $\beta=(\log\lambda)/2$,
 writing
\[
x=D_s \,\tilde x
\qquad
\text{ and }
\qquad
u=D_{s_u}\, \tilde u,
\]
where $\tilde x, \tilde u \in \ellipses_\beta$
and $s\ge 0, \;\, s_u\in\R$. If  $(x,u) \in \cone_m$, we  take
  $t \in I_m^+(x,u)$ and observe that $R(D_t \,x) > R(x)\ge \beta$. Then
 \cite[Lemma 4.3 (i)]{CCS2} can be applied,  giving
\begin{equation}
\label{tilde}
 |\tilde u-\tilde x|\lesssim  | u-D_t \,x|
\leq 2^m \sqrt t
\simeq 2^m
\sqrt{{|u-x|}/{|x|} }\,,
\end{equation}
the last step because of \eqref{u-x}.

We shall  cover the ellipsoid $E_\beta$ with little caps, and start with
$E_1$.
The small number  $\delta > 0$ will be specified below, depending only on
$n$, $Q$ and $B$.
 Define for $e \in E_1$ the cap
$\Omega_e^1 = E_1 \cap B(e,\delta).$
  We cover $E_1$ with caps
$\Omega_e^1$ with $e$ ranging over a finite subset of $ E_1$,
in such a way that the doubled caps $\tilde \Omega_e^1 = E_1 \cap B(e,2\delta)$
have $C$-bounded overlap.

Since $E_\beta = \sqrt\beta \, E_1$, we can scale these caps to get caps
\begin{equation*}
 \Omega_e^\beta = \sqrt\beta \, \Omega_e^1
=  E_\beta \cap B\left(\sqrt\beta \,e, \sqrt\beta \,\delta\right)
\end{equation*}
  covering $E_\beta$.
Similarly, $\tilde \Omega_e^\beta = \sqrt\beta \, \tilde \Omega_e^1$.

For each  $x \in \mathcal E_\lambda$,   the point
 $\tilde x$ will belong to some cap $\Omega_{e}^\beta$ of the covering.
In the proof of Lemma  \ref{tipodebole} we need only consider
 those $u$  for which $\tilde u$ is in  the doubled cap
$\tilde \Omega_{e}^\beta $.
The reason is that if $\tilde u \notin \tilde \Omega_{e}^\beta$, then
 $|\tilde u-\tilde x| \gtrsim \sqrt\beta\, \delta \simeq  |x|$, and  \cite[Lemma 4.3 (i)]{CCS2} implies   $| u-D_t\, x|  \gtrsim    |x|$ and also $| u-x|  \gtrsim    |x|$.
 This and the definition of
 $I_m(x,u)$ lead to $|x| \lesssim 2^m \sqrt t$  and thus $1+|x| \lesssim 2^m $.
 It follows that
 \begin{equation*}
   \mathcal{M}_{m} (x,u) \lesssim  \, e^{R( x)}\, 2^m\,
|x|^{{(n+1)}/2}\, |x-u|^{{(1-n)}/2}  \lesssim \,e^{R(x)}\, 2^m\, |x| \lesssim\, e^{R(x)}\, 2^{(n+3)m}\, (1+ |x|)^{-n-1}.
 \end{equation*}

Since the last expression is independent of $u$ and has integral
\begin{equation*}
   \int  e^{R(x)}\, 2^{(n+3)m}\, (1+ |x|)^{-(n+1)} \,d\gamma_\infty(x) \lesssim 2^{Cm},
 \end{equation*}
 this part of the kernel $\mathcal{M}_m$ defines an operator which is of strong type $(1,1)$, with the
 desired bound.

Thus we fix a cap $\Omega_{e}^\beta$, assuming that
$\tilde x \in \Omega_{e}^\beta$ and  $\tilde u \in \tilde \Omega_{e}^\beta$.
By means of a rotation, we may also assume that $ e$ is on the positive
$x_1$ axis.
Then we write  $\tilde x $ as
$\tilde x=(\tilde x_1, \tilde x') \in \R \times  \R^{n-1}$,
and similarly $\tilde u=(\tilde u_1, \tilde u')$.
If $\delta$ is chosen small enough, we will then have
 \begin{equation}  \label{transversal}
    |\tilde x-\tilde u| \simeq |\tilde x' -\tilde u'|,
 \end{equation}
essentially because the $x_1$ axis is transversal to $E_\beta$ at the point of
intersection $\sqrt\beta \:e$.
 Further, the area measure $dS_\beta$ of
 $E_\beta$ will satisfy
 \begin{equation}\label{**}
   dS_\beta(\tilde u) \simeq d\tilde u'
 \end{equation}
in $ \tilde \Omega_{e}^\beta$, again if  $\delta$ is small.

We now recall Proposition 8 in \cite{Sjogren-Li}. This proposition is also applied in another framework in \cite{BS}.

\begin{proposition}\cite{Sjogren-Li}
\label{Prop8-Sjogren-Li}
The operator
\begin{equation*}
 Tg(\xi)= e^{-2\xi_1}  \int_{\substack{\eta_1<\xi_1-1\\ |\xi'-\eta'|< \sqrt{\xi_1-\eta_1}}}
(\xi_1-\eta_1)^{{(1-n)}/{2}}
g(\eta)\, d\eta
\end{equation*}
maps
$L^1 (d\eta)$ boundedly into $L^{1,\infty}(e^{2\xi_1}\, d\xi)$. Here
$\xi=(\xi_1,\, \xi')\in\R\times\R^{n-1}$ and similarly for $\eta$.
\end{proposition}
In order to apply this result,
 we define new  variables
$\xi=(\xi_1, \xi') \in \R\times \R^{n-1}$ and analogously $\eta=(\eta_1, \eta')$, defined
for $x \in \mathcal E_\lambda$ and $(x,u) \in \mathcal C_m$ satisfying
$\tilde x \in \Omega_{e}^\beta$ and  $\tilde u \in \tilde \Omega_{e}^\beta$,
    by
\begin{align*}
\xi_1= -\frac12\, R(x), \qquad \eta_1= -\frac12\, R(u)
\end{align*}
and
\begin{align*}
\xi'= 2^{-m} \sqrt{ \log\lambda }\ \tilde x'
,\qquad \eta'= 2^{-m}  \sqrt{\log\lambda} \ \tilde u'.
\end{align*}

Lemma \ref{size} implies that
\begin{align}\label{Intervallo xi,eta}
 |u-x| \simeq (\xi_1-\eta_1)/ |x|  \simeq (\xi_1-\eta_1)/\sqrt{\log\lambda} .
\end{align}
Then
$\xi_1-\eta_1 \gtrsim A$ because of the global condition. Choosing $A$ large
enough, we will have
\begin{equation*}
 \xi_1-\eta_1 > 1.
\end{equation*}
Applying \eqref{transversal}, \eqref{tilde} and \eqref{Intervallo xi,eta},
we obtain
\begin{equation*}
  |\xi' - \eta'| = 2^{-m}\, \sqrt{\log\lambda} \:|\tilde u' -\tilde x'|
\simeq  2^{-m}\, \sqrt{\log\lambda} \:|\tilde u -\tilde x|
\lesssim \sqrt{ |x|\, |u-x|} \simeq  \sqrt{\xi_1-\eta_1}.
\end{equation*}
This allows us to estimate $\mathcal{M}_m$ in terms of the coordinates
$\xi$ and  $\eta$ :
\begin{equation*}
  \mathcal{M}_{m} (x,u) \lesssim
e^{-2\xi_1} \, (\log\lambda)^{n/2}\,(\xi_1-\eta_1)^{(1-n)/2}\,\chi_{\mathcal C'_m},
\end{equation*}
where
\begin{equation*}
  \mathcal C'_m=  \left\{(\xi,\eta ):\:  \xi_1-\eta_1 > 1, \;   |\xi' - \eta'|
\le C\, \sqrt{\xi_1-\eta_1}\right\}
\end{equation*}
for some $C$.

We must also express the Lebesgue measures $dx$ and $du$ in terms of
 $\xi$ and $\eta$ , with $x$ and $u$ restricted as before. By \eqref{def:leb-meas-pulita},
\begin{equation*}
  dx \simeq e^{-s\tr B}\,|x| \,ds\,dS_\beta(\tilde x)
\simeq  \sqrt{\log\lambda}\;ds\,dS_\beta(\tilde x),
\end{equation*}
the last step since $x \in \mathcal E_\lambda$ implies $s\lesssim 1$. Similarly,
 $ du \simeq \sqrt{\log\lambda}\;ds_u\:dS_\beta(\tilde u)$.

Because of  \eqref{vel-4}, we can write
$|{\partial \xi_1 }/{\partial s}|
= \left| {\partial R(D_s \,\tilde x)}/{\partial s}\right|/2
\simeq |D_s \,\tilde x|^2 = |x|^2 \simeq \log\lambda $, and if  $\tilde x$ is kept fixed, we will have
 $ds \simeq (\log\lambda)^{-1}\,d\xi_1$.
From \eqref{**} applied to $x$, we have
$dS_\beta(\tilde x) \simeq d\tilde x' = 2^{(n-1)m}\,(\log\lambda)^{(1-n)/2}
\,d\xi'$. Altogether, we get
\begin{equation}  \label{+}
  dx \simeq 2^{(n-1)m}\,(\log\lambda)^{-n/2}\,d\xi  \qquad \quad \mathrm{and}
\qquad  \quad  du \simeq 2^{(n-1)m}\,(\log\lambda)^{-n/2}\,d\eta.
\end{equation}
Letting $g(\eta) = e^{-R( u)}\,f(u)$,
we can summarize the above and write
\begin{align*}
  \int  \mathcal{M}_{m} (x,u)\, f(u)\,d\gamma_{\infty}(u)
    &\lesssim 2^{Cm}\, e^{-2\xi_1} \int_{\mathcal C'_m}
(\xi_1-\eta_1)^{(1-n)/2}\,g(\eta)\,d\eta.
\end{align*}
Hence, the set of points $x$ where
\begin{equation} \label{++}
   \int  \mathcal{M}_{m} (x,u)\, f(u)\,d\gamma_{\infty}(u) > \lambda
\end{equation}
is, after the change of coordinates, contained in the set of $\xi$ for which
\begin{equation}\label{Sjogren-Li-operator}
  e^{-2\xi_1}
\int_{\substack{\eta_1<\xi_1-1\\ |\xi'-\eta'|<C \sqrt{\xi_1-\eta_1}}}
(\xi_1-\eta_1)^{{(1-n)}/{2}}\,
g(\eta)\, d\eta \gtrsim  2^{-Cm}\, \lambda.
\end{equation}
The  integral here fits with that in Proposition  \ref{Prop8-Sjogren-Li},
except for the factor $C$ in the domain of integration. This factor can easily
be eliminated by means of a scaling of the variables $\eta'$.
Thus
 Proposition  \ref{Prop8-Sjogren-Li} tells us that the  level set defined by \eqref{++} has
$e^{2\xi_1}\,d\xi$
measure at most  $C\,2^{Cm}\,\lambda^{-1}\int g(\eta)\,d\eta$.
If we go back to the  $x$ coordinates,  \eqref{+} implies that
the $d\gamma_\infty$ measure of the same set is at
most
$$C\,2^{Cm}\,(\log\lambda)^{-n/2}\,\lambda^{-1}\int g(\eta)\,d\eta.
$$
But
$$
\int g(\eta)\,d\eta \simeq  2^{(1-n)m}\,(\log\lambda)^{n/2}
\int f(u)\,d\gamma_\infty(u),
$$
 again by  \eqref{+}.
Lemma  \ref{tipodebole} now follows.
\end{proof}

    Lemmata \ref{tipoforte} and \ref{tipodebole} together  imply  Proposition  \ref{Rm} and also Proposition  \ref{propo-glob-tsmall}.

\section{A counterexample   for $|\alpha|>2$
} \label{Counterex}

We  prove the ``only if'' part of Theorem  \ref{weaktype1}. Thus assuming
 $|\alpha|>2$, we will disprove the weak type $(1,1)$ of the
Riesz transform $ R^{(\alpha)}$.

The off-diagonal kernel of $ R^{(\alpha)}$ is
\begin{equation}\label{def:R-kernel}
\mathcal{R}_{\alpha} (x,u)
=\frac{1}{\Gamma (|\alpha|/2)}
\int_0^{+\infty} t^{(|\alpha|-2)/2}\,
D^{\alpha}_x\, K_t (x,u) \, dt,
\end{equation}
$K_t$ being  the Mehler  kernel as in \eqref{defKR}.

Repeated application of \eqref{der-1-Kt} in
Lemma \ref{derivative-kernel} implies that the derivative
 $D^{\alpha}_x\, K_t (x,u)$ is a sum of products of the form
$K_t (x,u)\,P(t,x,u)\,Q(t)$, where $P(t,x,u)$ is a product of factors of type
$P_j(t,x,u)$, and $Q(t)$ is a product of factors of type  $\Delta_{ ij} (t)$.
Since  $\Delta_{ ij} (t)$ does not depend on $x$, there will be nothing more.
More precisely, consider a term in this sum where the derivatives falling
on  $ K_t (x,u)$ are given by a multiindex $\kappa$, with
$\kappa \le \alpha$ in the sense of componentwise inequalities.
Then  $|\alpha| - |\kappa|$ differentiations must fall on the $P_j(t,x,u)$
factors, and necessarily $|\alpha| - |\kappa| \le |\kappa|$. This tells us
that $Q(t)$ must consist of
$|\alpha| - |\kappa|$ factors and also that  $P(t,x,u)$ consists of
$N := |\kappa| - (|\alpha| - |\kappa|)$ factors. It follows that
 $|\alpha| - |\kappa| =(|\alpha| - N)/2 $.  Thus we get products
 \begin{equation}
\label{factors}
   K_t (x,u)\,P^{(N)}(t,x,u)\:Q^{((|\alpha| - N)/2)}(t),
 \end{equation}
where the superscripts indicate the number of factors.  Since $|\kappa|$
can be any integer satisfying $ |\alpha|/2 \le |\kappa| \le  |\alpha|$,
 we see that $N$ runs over the set of integers in
$[0, |\alpha|]$ congruent with  $|\alpha|$ modulo 2.

With $\eta>0$  large, define
\begin{equation*}
u_0=Q_\infty(\eta, \ldots, \eta)\in\R^n,
\end{equation*}
where we mean the product of the matrix $Q_\infty$ and $(\eta, \ldots, \eta)$
written as a column vector.

Our $f$ will be $\delta_{u_0}$,
 and we will verify that the $L^{1,\infty}(\gamma_\infty)$ quasinorm of $R^{(\alpha)}\, f$ tends to $+\infty$ with $\eta$. Since  $\delta_{u_0}$ can
 be approximated by $L^1$ functions in a standard way, this will disprove the weak type $(1,1)$ estimate for  $R^{(\alpha)}$. We have
  \[R^{(\alpha)}\, f (x) =R^{(\alpha)}\, \delta_{u_0} (x)
 = \mathcal{R}_{\alpha} (x,u_0).\]
 
For reasons that will become clear below, we  fix a number
$t_0 \in (0,1/2)$, independent of $\eta$ and so small that
\begin{equation}\label{t0}
\left\langle (1, 1,\dots,1),\,  e^{t_0 B}\,e_j\right\rangle > 1/2, \qquad j = 1,\dots, n.
\end{equation}

Define $x_0 = D_{-t_0}\, u_0$.
 We are going to evaluate $R^{(\alpha)}\, \delta_{u_0} (x)$
when $x$ is in the ball $B\left(x_0,\sqrt t_0\right)$.  Then we will have
$|x| \simeq  |x_0| \simeq  |u_0| \simeq \eta$.

From \eqref{litet} we get an estimate of $K_t(x,u_0)$ for $0<t<1$. There we
 want the exponent
$|u_0-D_t\, x|^2/t$ to stay bounded when $x \in B\left(x_0,\sqrt {t_0}\right)$
and $t$ is close to $t_0$. Write
\begin{equation*}
  u_0-D_t \,x = u_0-D_t \,x_0 + D_t\, (x_0 - x)
= u_0 -D_{t-t_0}\, u_0 + D_t\, (x_0 - x),
\end{equation*}
which we must then  make smaller than constant times
$\sqrt t \simeq \sqrt {t_0} $.
Here
$$
|u_0 -D_{t-t_0}\, u_0| \simeq |t-t_0| \,|u_0|,
$$
because of Lemma \ref{differ}.  Thus we take $t$ with
$|t-t_0| < \sqrt {t_0} /|u_0| $, which implies
 $ t \simeq  t_0 $  for large enough $\eta$. Further,
 $|D_t\, (x_0 - x)|\simeq |x_0-x| < \sqrt {t_0}$.
Then $|u_0-D_t \,x | \lesssim  \sqrt t$, and it follows that
\begin{equation}
  \label{kt}
  K_t(x,u_0) \simeq {e^{R(x)}}\,{t_0^{-n/2}} \qquad \mathrm{if}\qquad
x \in B\big(x_0,\sqrt {t_0}\big) \qquad \mathrm{and}\qquad
|t-t_0| < \sqrt {t_0} /|u_0|.
\end{equation}

Lemma \ref{derivative-kernel} says that
\begin{equation}\label{pj}
   P_j(t,x,u_0)=
\left\langle
Q_\infty^{-1}x,e_j\right\rangle
+
{
\left\langle   Q_t^{-1}
e^{tB}\, e_j
  \,,\, u_0-D_t\, x\right\rangle}.
\end{equation}
The first summand here is for $x \in B\big(x_0,\sqrt t_0\big)$
\begin{align*}
  \left\langle Q_\infty^{-1}x,e_j\right\rangle
= &\, \left\langle Q_\infty^{-1} D_{-t_0}\,u_0,e_j\right\rangle +
\left\langle Q_\infty^{-1}(x-x_0), e_j\right\rangle\\  = & \,
 \left\langle e^{t_0 B^*} Q_\infty^{-1} u_0, e_j\right\rangle +
{\bigO}(|x-x_0|) \\
 =& \,   \left\langle (\eta, \eta,\dots, \eta),  e^{t_0 B}\,e_j\right\rangle +
{\bigO}(\sqrt {t_0}),
\end{align*}
for large $\eta$, by the definition of $u_0$.
Because of  \eqref{t0}, this leads to
\begin{equation}\label{largeterm}
    \langle Q_\infty^{-1}x,e_j\rangle \simeq  \eta  \simeq|x|,
\end{equation}
and we observe that $\langle Q_\infty^{-1}x,e_j\rangle$ does not depend on $t$.

Next, we
rewrite the product \eqref{factors} by using  \eqref{pj} to
expand the factor $P^{(N)}$.
  We will then get a sum of terms  like \eqref{factors}
but where $P^{(N)}$ is replaced by a product of  powers of
the two summands in  \eqref{pj}. For
 $N = |\alpha|$ one of the terms in this  sum will be
\begin{equation}  \label{pos}
  K_t (x,u_0)\, \prod_{j=1}^n\,  \left\langle Q_\infty^{-1}x,e_j\right\rangle^{\alpha_j}
\gtrsim  K_t (x,u_0)\,|x|^{|\alpha| },
\end{equation}
the inequality coming from \eqref{largeterm}. Since  $N = |\alpha|$, the corresponding factor $Q^{((|\alpha| - N)/2)}(t)$  is 1.
The positive quantity in \eqref{pos}   will give
the divergence we need for the counterexample. We  have to estimate the
absolute values of all the other terms.

To do so, let  $t \in (0,1)$.
For the second summand in   \eqref{pj}, we have
\begin{equation*}
  \left|\left\langle   Q_t^{-1}
e^{tB}\, e_j
  \,,\, u_0-D_t\, x\right\rangle\right| \lesssim \frac{|u_0-D_t\, x|}t ,
\end{equation*}
and by \eqref{est-for-Delta}
\begin{equation*}
 |\Delta_{ ij} (t)| \lesssim 1/t.
\end{equation*}
Thus each of the  terms we must estimate is controlled by an expression of type
\begin{equation}\label{prod}
  K_t (x,u_0)\,  |x|^{N_1}\, \left(\frac{|u_0-D_t \,x|}t\right)^{N_2}\,
\frac1{t^{(|\alpha|-N_1-N_2)/2}},
\end{equation}
where $N_1$ and  $N_2$ are nonnegative integers satisfying
 $N_1 + N_2 = N \le|\alpha| $ and $N_1 \le|\alpha|-1 $.
If instead of  $K_t (x,u_0)$ we plug in here the upper bound in
\eqref{litet} and reduce
slightly the coefficient $c$ in the exponential,
we can replace each factor $|u_0-D_t\, x|/t$ in \eqref{prod} by $1/\sqrt t$.
The quantity  \eqref{prod} is thus less than constant times
\begin{equation}\label{otherterns}
    { e^{R( x)}}\,{t^{-n/2}}\, \exp\left(-c\,\frac{|u_0-D_t \,x |^2}t\right)
 |x|^{N_1}\,\frac1 {t^{(|\alpha|-N_1)/2}}.
\end{equation}

We are now ready to estimate the integral in \eqref{def:R-kernel},
at first taken only over the interval $(0,1)$. Here $u=u_0$ and
 $x \in B\left(x_0,\sqrt t_0\right)$.  The positive term described
 in \eqref{pos} will, because of
\eqref{kt}, give a contribution which is larger than a constant $c$ times
\begin{align}
  |x|^{ |\alpha|}\, \int_0^{1} t^{(|\alpha|-2)/2} K_t (x,u) \, dt
\ge \, \, |x|^{ |\alpha|}\, \int_{|t-t_0| < \sqrt t_0 /|u_0|} t^{(|\alpha|-2)/2} K_t (x,u) \, dt\notag \\
\label{largeint}
 \gtrsim\,  |x|^{ |\alpha|}\, \,{e^{R(x)}}\,{t_0^{-n/2}}\, t_0^{(|\alpha|-1)/2}
\, |u_0|^{-1}  \: \simeq \: e^{R(x)}\, |x|^{ |\alpha|-1},
\end{align}
since $t_0 \simeq 1$ is fixed.

 Next, we consider the expression in \eqref{otherterns}.
 The corresponding part of the integral in  \eqref{def:R-kernel}
 will  be at most a constant  C  times
\begin{equation*}
  {e^{R(x)}}\,|x|^{ N_1}\, \int_0^{1} t^{(N_1 - n-2)/2}\,
\exp\left(-c\,\frac{|u_0-D_t \, x |^2}t\right) \, dt,
\end{equation*}
In order to estimate this integral, we write, recalling that
 $D_t\,{x_0}\, = D_{t-t_0}\,u_0 $,
\begin{equation*}
  |u_0-D_t\, x | \ge |u_0-D_{t-t_0}\,u_0| - |D_t \,(x-x_0)|.
\end{equation*}
The first summand here satisfies  for $0<t<1$, in view of Lemma \ref{differ},
\begin{equation*}
  |u_0-D_{t-t_0}\,u_0|  \simeq |t - t_0|\, |u_0|,
\end{equation*}
and the second summand is controlled by $\sqrt t_0 $.
Thus if $ |t - t_0| > C/|u_0|$ for some large $C$, we will have
\begin{equation*}
  |u_0-D_t\, x | \ge  |t - t_0|\, |u_0| \simeq 1 +  |t - t_0|\, |u_0|,
\end{equation*}
so that
\begin{equation*}
  \frac{|u_0-D_t \,x |^2}t  \gtrsim \frac1{t} + \frac{ |t - t_0|^2\, |u_0|^2}t.
\end{equation*}
This implies that
\begin{align*}
 & {e^{R(x)}}\,|x|^{ N_1}\,
\int_{\substack{ |t - t_0| > C/|u_0|\\ 0<t<1}} t^{(N_1 -n-2)/2}\,
\exp\left(-c\,\frac{|u_0-D_t\, x |^2}t\right) \, dt\\
\lesssim\, &\, {e^{R(x)}}\,|x|^{ N_1}\,
\int_0^1 t^{(N_1-n-2)/2}\,\exp\left(-\frac c t\right)\,
\exp\left(-c\,\frac{ |t - t_0|^2\, |u_0|^2}t\right)\, dt\\
\lesssim\, &\,{e^{R(x)}}\,|x|^{ N_1}\,\int_\R \exp{( -c\,|t - t_0|^2\, |u_0|^2})\, dt\\
\lesssim \,&\,{e^{R(x)}}\,|x|^{ N_1}\,\frac1{|u_0|} \:  \simeq  \: {e^{R(x)}}\,|x|^{ N_1-1}.
\end{align*}
What remains is
\begin{align*}
  {e^{R(x)}} & \,|x|^{ N_1}\,
\int_{ |t - t_0| < C/|u_0|}\, t^{(N_1-n-2)/2}\,
\exp\left(-c\,\frac{|u_0-D_t x |^2}t\right) \, dt\\
\lesssim\, &\, {e^{R(x)}}\,|x|^{ N_1}\,  t_0^{(N_1 -n-2)/2}\,|u_0|^{-1}
\: \simeq \:  {e^{R(x)}}\,|x|^{ N_1-1}.
\end{align*}
Since $N_1 < |\alpha|$, the last expression is less than
${e^{R(x)}}\,|x|^{|\alpha|-2 }$, and we see that
 for large $\eta$ the positive expression in
\eqref{largeint} dominates over the effects of the other terms.

We finally treat the integral over $t>1$.
For $x \in B\left(x_0,\sqrt{t_0}\right)$ and $t>1$,
  \eqref{tstort},   \eqref{est-for-Pj} and  \eqref{est-for-Delta}
 imply the following three estimates
\begin{equation*}
  K_t (x,u_0)
\lesssim e^{R(x)}
\exp
\Big[
-\frac12
\left|
D_{-t}\,u_0- x\right|_Q^2
\Big],
\end{equation*}
\begin{equation*}
 |P_j (t,x,u_0)|\lesssim
e^{-ct}\, | D_{-t} \,  u_0-x| + |D_{-t} \,  u_0|
\end{equation*}
and
 \begin{equation*}
|\Delta_{ij} (t)| \lesssim
{e^{-ct}}.
\end{equation*}
We can delete the factor $| D_{-t}\,   u_0-x|$ from the second of these formulas,
if we reduce slightly the coeffient $1/2$ in the first formula. Further,
\begin{equation}\label{9a}
| D_{-t} \,  u_0-x| \ge  |D_{t_0-t} \,x_0 - x_0|-|x_0-x|.  
\end{equation}
An argument like \eqref{x-Dtx} now leads to
 $ |D_{t_0-t} \,x_0 - x_0| \gtrsim |x|$,
because here  $t_0-t < -1/2$ and so \eqref{vel-4} implies that
$|x_0|_Q^2 -|D_{t_0-t} \,x_0|_Q^2  \simeq |x_0|_Q^2$.
Since  $|x_0-x|$ is much smaller than $|x|$, we conclude from \eqref{9a}
that $| D_{-t} \,  u_0-x| \simeq |x| $.
Moreover,
$| D_{-t}\,   u_0| \lesssim e^{-ct}\, |u_0| \simeq  e^{-ct}\, |x|$ by Lemma
\ref{expsB-bounded}.   Estimating the products in \eqref{factors},  we arrive at
\begin{equation*}
  |D^\alpha_x\, K_t(x,u_0)| \lesssim e^{R(x)}\, \exp{(-c\,|x|^2)}\: e^{-ct}\, |x|^C,
\qquad t>1.
\end{equation*}
 Hence,
\begin{equation*}
\int_1^\infty t^{(|\alpha|-2)/2}\, |D^\alpha K_t(x,u_0)|\, dt \lesssim e^{R(x)},
\end{equation*}
and this is much smaller than the quantity in \eqref{largeint}.

Summing up, we get an estimate for the integral in \eqref{def:R-kernel}
saying that
\begin{equation*}
  \mathcal{R}_{\alpha} (x,u_0)
\gtrsim e^{R(x)}\, |x|^{ |\alpha|-1}, \qquad  x \in B\big(x_0, \sqrt t_0\big).
\end{equation*}
Let $\lambda =  e^{R(x_0)}\, |x_0|^{ |\alpha|-1} $.
The ball $B\big(x_0, \sqrt t_0\big)$ contains the set
\begin{equation*}
V_{x_0} =  \{x = D_s \, \tilde x: R( \tilde x) = R(x_0), \; \;
|\tilde x-x_0| < c,\; \; 0<s<c/|x_0|^2\}
\end{equation*}
for some $c$. Then  $e^{R(x)} \simeq e^{R(x_0)} $
 in $V_{x_0}$ as follows from  \eqref{stime per R-t(x)-R(x)},
 and so $\mathcal{R}_{\alpha} (x,u_0)
\gtrsim \lambda$ in $V_{x_0}$.
 From \eqref{def:leb-meas-pulita} we see that the
measure of $V_{x_0}$ is
\begin{align*}
  \gamma_\infty(V_{x_0}) =& \int_0^{c/|x_0|^2} \int_{|\tilde x-x_0| < c} e^{-R(D_s \, \tilde x)}\,
e^{-s\tr B}\, \frac{ |Q^{1/2}\, Q_\infty^{-1} \tilde x |^2}
{2\,| Q_\infty^{-1} \tilde x  |}\,
 dS_\beta( \tilde x)\,ds\,\\
\simeq &\, e^{-R(x_0)} \int_0^{c/|x_0|^2} |x_0|\,ds
\:\simeq \: e^{-R(x_0)} \,|x_0|^{-1}.
\end{align*}
We find that
\begin{equation*}
  \lambda\, \gamma_\infty(V_{x_0}) \gtrsim  |x_0|^{ |\alpha|-2}.
\end{equation*}
Since $|\alpha| > 2$, this expression
 tends to $+\infty$ with  $\eta$,
and so does the  $L^{1,\infty}(\gamma_\infty)$ quasinorm of $R^{(\alpha)}\, f$.

Theorem \ref{weaktype1} is completely proved.

\end{document}